\documentclass[11pt]{CGC2}
\usepackage[english]{babel}
\usepackage{verbatim}
\usepackage{psfrag}
\usepackage{graphicx}
\usepackage{textcomp}
\usepackage{xcolor}
\usepackage{makeidx}  % allows for indexgeneration
\usepackage{graphicx}
\usepackage{algorithmic}
\usepackage{algorithm}
\usepackage{amsmath}
\usepackage{pstricks}

\usepackage{latexsym}

\usepackage{booktabs}
\usepackage{verbatim}
\usepackage{psfrag}
\usepackage{graphicx}     
\usepackage{xcolor}
\usepackage{graphicx}
\usepackage{amsmath}
\usepackage{amsfonts}
\usepackage{amssymb}
\usepackage{mathrsfs}
\usepackage{rotating}
\usepackage{multirow}

\newcommand{\F}{\mathbb{F}_2}
\newcommand{\FQ}{\mathbb{F}_{q^2}}
\newcommand{\Fq}{\mathbb{F}_q}
\newcommand{\Tr}[3]{\mathrm{Tr}^{#1}_{#2}(#3)}
\newcommand{\Tra}{\mathrm{Tr}}
\newcommand{\N}[3]{\mathrm{N}^{#1}_{#2}(#3)}
\newcommand{\No}{\mathrm{N}}
\newcommand{\He}{\mathcal{H}}
\newcommand{\I}{\mathrm{Im}}
\newcommand{\Bmq}{\mathcal{B}_{m,q}}

\begin{document}

%%%%%%%%%%%%%%%%%%%%%%%%%%%%%%%%%%%%%%%%%%%%%%%%%%%%%%%%%%%%%%%%%%%%%

\begin{frontmatter}

\title{ On the Hermitian curve, its intersections with some conics and their
applications to affine-variety codes and Hermitian codes } 
\runtitle{Hermitian curves, intersection with conics, affine-variety codes.}

{\author{Chiara Marcolla}} {\tt{(chiara.marcolla@unitn.it)}}\\
{Department of Mathematics, University of Trento, Italy}

{\author{Marco Pellegrini}} {\tt{(pellegrini@mail.dm.unipi.it)}}\\
{Department of Mathematics, University of Pisa, Italy}

{\author{Massimiliano Sala}} {\tt{(maxsalacodes@gmail.com)}}\\
{Department of Mathematics, University of Trento, Italy}

\runauthor{C.~Marcolla, M.~Pellegrini, M.~Sala}

%%%%%%%%%%%%%%%%%%%%%%%%%%%%%%%%%%%%%%%%%%%%%%%%%%%%%%%%%%%%%%%%%%%%%%%%%%%%%%

\begin{abstract}
For any affine-variety code we show how to construct an ideal
whose solutions correspond to codewords with any assigned weight.
We classify completely the intersections of the Hermitian curve with
lines and parabolas (in the $\FQ$ affine plane).
Starting from both results, we are able to obtain geometric
characterizations for
small-weight codewords for some families of Hermitian codes over any
$\FQ$. From the geometric characterization,
we obtain explicit formulae.
In particular, we determine the number of minimum-weight codewords for all
Hermitian codes with $d\leq q$ and  all second-weight codewords for
distance-$3,4$ codes.
\end{abstract}
 
%%%%%%%%%%%%%%%%%%%%%%%%%%%%%%%%%%%%%%%%%%%%%%%%%%%%%%%%%%%%%%%%%%%%%%%%%%%%%%


%%%%%%%%%%%%%%%%%%%%%%%%%%%%%%%%%%%%%%%%%%%%%%%%%%%%%%%%%%%%%%%%%%%%%%%%%%%%%%

\end{frontmatter}

%==================================================================

\section{Introduction}
\label{intro}

For any $q$, the \textit{Hermitian curve} $\He$ is the planar curve
defined over $\FQ$ by the affine equation $x^{q+1}=y^q+y$.

This curve has genus $g=\frac{q(q-1)}{2}$ and has $q^3$ rational affine
points, plus one point
at infinity, so it has $q^3+1$ rational points over $\FQ$ and
therefore it is a maximal curve
\cite{CGC-alg-art-rucsti94}.
This is the best known example of maximal curve and there is a vast
literature on its properties, see \cite{CGC-cod-book-hirschfeld2008algebraic} for a recent survey.
Moreover, the Goppa code \cite{CGC-cd-book-goppa} constructed on this curve
is by far the most studied, due to the simple basis of its
Riemann-Roch space \cite{CGC-cd-book-stich}, which can be written explicitly. The Goppa construction has been
generalized in \cite{CGC-cd-art-lax} to the so-called
affine-variety codes.

Although a lot of research has been devoted to geometric properties of
$\He$, here we present in this paper a  classification result,
providing for any $q$ the number of possible intersection points
between any parabola and $\He$. Moreover, we can characterize precisely
the parabolas obtaining a given intersection number and so we can count them.

In this paper we also provide an algebraic and geometric description
for codewords of a given weight belonging to any fixed
affine-variety codes. The specialization of our results to the
Hermitian case permits us to give explicit formulae for the number
of some small-weight codewords.
We expand on our 2006 previous result \cite{CGC-cod-misc-salpell06}, where we proved the intimate connection between curve intersections and minimum-weight codewords.

The paper is organized as follows:
\begin{itemize}
\item In Section \ref{pre} we provide our notation, our
first preliminary results on the algebraic characterization of
fixed-weight codewords of any-affine variety codes and some easy results
on the intersection between the Hermitian curve and any line.

\item In Section \ref{parabole}  the main result is Theorem~\ref{teo.principe}, where we provide a complete classification of
intersections between $\He$ and any parabola $y=ax^2+bx+c$ ($a\not=0$).
The proof is divided in three main parts. In the first part of Section
\ref{parabole} we lay down some preliminary lemmas and we skecth our
proving argument, that is, the use of the authomorphism group for
$\He$. In Subsection~\ref{odd} we deal with the odd-characteristic
case and in Subsection~\ref{even} we deal with the even-characteristic case.

\item In the beginning of Section \ref{codicihermitiani} we provide a
division of Hermitian codes in four phases, which is a slight
modification of the division in \cite{CGC-cd-book-AG_HB}, and we give
our algebraic characterization of fixed-weight codewords of any Hermitian codes.
We study in deep the first phase (that is, $d\leq q$)
in Subsection~\ref{parolemin} and we use
these results  to completely classify geometrically the minimum-weight codewords
for all first-phase codes in Subsection~\ref{minimum}.
In Subsection~\ref{secondweight} we can count some special configurations 
of second weight codewords for any first-phase code and finally in
Subsection \ref{complete} we can count the exact number of
second-weight codewords for the special case when $d=3,4$.
Most results in this section rely on our results on intersection
properties of $\He$ (with some special conics).

\item In Section~\ref{comp} we show how we computed specific examples
and tested experimentally all our counting results in several
finite fields. We provide all details to let the reader check
our computations.

\item In Section \ref{conc} we draw some conclusions and propose some
open problems.

\end{itemize}

%==================================================================

\section{Preliminary results}
\label{pre}
\subsection{Known facts on Hermitian curve and Affine-variety code}

From now on we consider $\FQ$ the finite field with $q^2$ elements, and $\Fq$ the finite field
with $q$ elements, where $q$ is a power of a prime. We call $\alpha$ a (fixed) primitive element of $\FQ$, and
we consider $\beta=\alpha^{q+1}$ as a primitive element of $\Fq$. From now on $q,q^2,\alpha$ and $\beta$ are understand.

For any $q$, the \textit{Hermitian curve} $\He=\He_q$ is the curve
defined over $\FQ$ by the affine equation
\begin{equation}\label{eqHerm}
  x^{q+1}=y^q+y
\end{equation}
This curve has genus $g = \frac{q(q-1)}{2}$ and has $n = q^3$ rational affine
points, denoted by $P_1,\ldots,P_n$. For any $x\in\FQ$, the equation (\ref{eqHerm})
has exactly $q$ distinct solutions in $\FQ$. The curve contains also one point
at infinity $P_{\infty}$, so it has $q^3+1$ rational points over $\FQ$ \cite{CGC-alg-art-rucsti94}.\\

Let $k \geq 1$. For any ideal $I$ in a polynomial ring $\Fq[X]$, where
$X=\{x_1,\dots,x_k\}$, we denote by $\mathcal{V}(I)\subset (\overline{\mathbb{F}}_q)^k$ its variety, that is, the set of its common roots. For any $Z\subset
(\overline{\mathbb{F}}_q)^k$ we denote by $\mathcal{I}(Z)\subset \Fq[X]$ the vanishing ideal of $Z$, that is, $\mathcal{I}(Z)=\{f\in \Fq[x]\mid f(Z)=0\}$.\\

Let $g_1,\ldots,g_s \in \Fq[X]$, we denote by $I=\langle g_1,\ldots,g_s\rangle$
the ideal generated by the $g_i$'s. Let $\{x_1^q-x_1,\ldots,x_k^q-x_k\}\subset I$.
Then $I$ is zero-dimensional and radical \cite{CGC-alg-art-seidenberg1}. Let $\mathcal{V}(I)=\{P_1,\dots,P_n\}$.
We have an isomorphism of $\Fq$ vector spaces (an \textit{evaluation map}):
\begin{equation}\label{eval}
\begin{array}{rccl}
  \phi:R= &\Fq[x_1,\dots,x_k]/I & \longrightarrow & (\Fq)^n\\
  & f & \longmapsto & (f(P_1),\dots,f(P_n)).
\end{array}    
\end{equation}
Let $L \subseteq R$ be an $\Fq$ vector subspace of $R$ with dimension $r$.

\begin{definition}
The {\bf affine--variety code} $C(I,L)$ is the image $\phi(L)$ and the
affine--variety code $C^{\perp}(I,L)$ is its dual code.
\end{definition}
Our definition is slightly different from to that in \cite{CGC-cd-art-lax}
and follows instead that in \cite{CGC-cd-prep-manumaxchiara12}.
Let $L$ be linearly generated by $b_1,\dots,b_r$ then the matrix
$$H=\left(
\begin{array}{cccc}
b_1(P_1) & b_1(P_2) & \dots & b_1(P_n)\\
\vdots & \vdots & \cdots & \vdots\\
b_r(P_1) & b_r(P_2) & \dots & b_r(P_n)
\end{array}
\right)$$
is a generator matrix for $C(I,L)$ and a parity--check matrix for $C^{\perp}(I,L)$. 

For more recent results on affine-variety codes see \cite{CGC-cd-book-Geil08,CGC-cd-prep-manumaxchiara12,CGC-cd-art-lax11}.

\subsection{First results on words of given weight}

Let $0\leq w\leq n$ and $C$ be a linear code. The 
$A_w(C)=|\{c\in C \mid \mathrm{w}(c)=w\}|$.\\
Let $\bar{z}\in(\Fq)^n$, $\bar{z}=(\bar{z}_1,\ldots,\bar{z}_n)$. Then
\begin{equation}\label{eqtot}
  \bar{z}\in C(I,L)^{\perp}\iff H\bar{z}^{T}=0\iff\sum_{i=1}^n\bar{z}_i b_j(P_i)=0 \quad j=1,\dots,r
\end{equation}
\begin{proposition}[\cite{CGC-cod-misc-pell06}]
\label{proptot}
Let $1\leq w\leq n$.\\ Let $J_w$ be the ideal in $\Fq[x_{1,1},\ldots x_{1,k},\ldots,x_{w,1},\ldots x_{w,k},z_1,\ldots,z_w]$ generated by 
\begin{align}\label{e1}
& \sum_{i=1}^wz_ib_j(P_i) \quad j=1,\ldots, r\\
\label{e2}
& g_h(x_{i,1},\ldots x_{i,k}) \quad i=1,\dots,w \textrm{ and } h=1,\ldots,s\\
\label{e3}
& z_i^{q-1}-1\quad i=1,\dots,w \\
\label{e4}
& \prod_{1 \leq l \leq k}((x_{j,l}-x_{i,l})^{q-1}-1) \quad 1 \leq j<i \leq w.
\end{align}
Then any solution of $J_w$ corresponds to a codeword of $C^{\perp}(I,L)$ with weight $w$. Moreover,
$$A_w(\mathcal{C}^{\perp}(I,L))=\frac{|\mathcal{V}(J_{w})|}{w!}.$$
\end{proposition}
\begin{proof}
Let $\sigma$ be a permutation, $\sigma\in S_n$. It induces a permutation
$\hat{\sigma}$ acting over $\{x_{1,1},\ldots, x_{1,k},\ldots,x_{w,1},\ldots x_{w,k},z_1,\ldots,z_w\}$
as $\hat{\sigma}(x_{i,l})=x_{\hat{\sigma}(i),l}$ and $\hat{\sigma}(z_{i})=z_{\hat{\sigma}(i)}$.
It is easy to show that $J_w$ is invariant w.r.t. any $\hat{\sigma}$,
since each of (\ref{e1}), (\ref{e2}), (\ref{e3}) and (\ref{e4}) is so.\\

Let $Q=(\overline{x}_{1,1},\ldots \overline{x}_{1,k},\ldots,\overline{x}_{w,1},\ldots
\overline{x}_{w,k},\overline{z}_1,\ldots,\overline{z}_w)\in\mathcal{V}(J_w)$.
We can associate a codeword to $Q$ in the following way.
For each $i=1,\ldots,w$, $P_{r_i}=(\overline{x}_{i,1},\ldots\overline{x}_{i,k})$
is in $\mathcal{V}(I)$, by (\ref{e2}). We can assume $r_1<r_2<\ldots<r_w$, via a permutation
$\hat{\sigma}$ if necessary. Note that (\ref{e4}) ensures that for each $(i,j)$,
with $i\ne j$, we have $P_{r_i}\ne P_{r_j},$ since there is a $l$ such that
$x_{i,l}\ne x_{j,l}$. Since $\overline{z}_i^{q-1}=1$ (\ref{e3}),
$\overline{z}_i\in \Fq\setminus\{0\}$. Let $c\in(\Fq)^n$ be
$$c=(0,\ldots,0,\underset{\underset{P_{r_1}}{\uparrow}}{\overline{z}_{1}},
0,\dots,0,\underset{\underset{P_{r_i}}{\uparrow}}{\overline{z}_i},0,\dots,0,
\underset{\underset{P_{r_w}}{\uparrow}}{\overline{z}_w},0,\dots,0).$$
We have that $c\in\mathcal{C}^{\perp}(I,L)$, since (\ref{e1}) is equivalent to (\ref{eqtot}).

\noindent Reversing the previous argument, we can associate to any codeword a solution
of $J_w$. By invariance of $J_w$, we actually have $w!$ distinct solutions for any codeword.
So, to get the number of codewords of weight $w$, we divide $|\mathcal{V}(J_{w})|$ by $w!$.
\end{proof}
Note that this approach is a generalization of the approach in \cite{CGC-cd-art-short}
to determinate the number of words having given weight for a cyclic code.

\subsection{Intersection between the Hermitian curve $\He$ and a line}
\label{rette}

We consider the norm and the trace, the two functions defined as follows. 
\begin{definition}\label{nt}
The \textbf{norm} $\mathrm{N}^{\mathbb{F}_{q^m}}_{\Fq}$ and the \textbf{trace}
$\mathrm{Tr}^{\mathbb{F}_{q^m}}_{\Fq}$ are two functions from $\mathbb{F}_{q^m}$ to
$\Fq$ such that
$$\N{\mathbb{F}_{q^m}}{\Fq}{x}=x^{1+q+\dots+q^{m-1}} \qquad
\Tr{\mathbb{F}_{q^m}}{\Fq}{x}=x+x^{q}+\dots+x^{q^{m-1}}$$
\end{definition}%

We denote with N and Tr, respectively, the norm and the trace from $\FQ$ to $\Fq$.
It is clear that $\He=\{\No(x)=\Tra(y)\mid x,y\in\FQ\}$.\\
We can define a similar curve $\mathcal H'=\{\No(x)=-\Tra(y)\mid x,y\in\FQ\}$ and, using the next lemma,
it is easy to see that also  $\mathcal H'$ contains $q^3$ affine rational points.
A well-known fact is the following \cite{CGC-cd-book-niederreiter}.
\begin{lemma}\label{solNTR}
For any $t\in\Fq$, the equation $\Tra(y)=y^q+y=t$ has exactly $q$ distinct solutions
in $\FQ$. The equation $\No(x)=x^{q+1}=t$ has exactly $q+1$ distinct solutions,
if $t\ne 0$, otherwise it has just one solution.
\end{lemma}
\begin{proof}
The trace is a linear surjective function between two $\Fq$-vector spaces of dimension,
respectively, $2$ and $1$. Thus, $\dim(\ker(\Tra))=1$, and this means that for any
$t\in\Fq$ the set of solutions of the equation $\Tra(y)=y^q+y=t$ is non-empty and then
it has the same cardinality of $\Fq$, that is, $q$.

The equation $x^{q+1}=0$ has obviously only the solution $x=0$. If $t\ne 0$, since
$t\in\Fq$, we can write $t=\beta^i$, so that $x=\alpha^{i+j(q-1)}$ are all solutions.
We can assign $j=0,\ldots,q$, and so we have $q+1$ distinct solutions.
\end{proof}

\begin{lemma}\label{interette} Let $\mathcal{L}$ be any vertical line $\{x=t\}$,
with $t\in\FQ$. Then $\mathcal{L}$ intersects $\He$ in $q$ affine points.
\end{lemma}
\begin{proof}
For any $t\in\FQ$, $t^{q+1}\in\Fq$, and so the equation $y^q+y=t^{q+1}$
has exactly $q$ distinct solutions by applying Lemma \ref{solNTR}.
\end{proof}

\begin{lemma}\label{intertot}
In the affine plane $(\FQ)^2$, the total number of non-vertical lines is $q^4$. Of these,
$(q^4-q^3)$ intersect $\He$ in $(q+1)$ points and $q^3$ are tangent to $\He$, i.e. they
intersect $\He$ in only one point. 
\end{lemma}
\begin{proof}
Let $\mathcal{L}$ any non-vertical line, then $\mathcal{L}=\{y=ax+b\}$, with $a,b\in\FQ$. We have $q^2$ choices
for both $a$ and $b$, so the total number is $q^4$. Then
$$\He\cap\mathcal{L}=\{(x,ax+b)\mid a^qx^q+b^q+ax+b=x^{q+1},\,\, x\in\FQ\}.$$
Let $c=c(a,b)=a^{q+1}+b^q+b$, then $c\in\Fq$. We have two distinct cases:
\begin{itemize}
  \item $c=0$. Then $a^qx^q+b^q+ax+b=x^{q+1}$ becomes $a^qx^q-a^{q+1}+ax=x^{q+1}$,
  which gives $x=a^q$, that is, $\mathcal{L}$ is tangent.
  \item $c\ne 0$. Then $a^qx^q+b^q+ax+b=x^{q+1}$ becomes $x^{q+1}-a^qx^q+a^{q+1}-ax=c$,
  which gives $(x-a^q)^{q+1}=c$. Since $c=(\alpha^{q+1})^r$ for $1\leq r\leq q-1$,
  we have $x=a^q+\alpha^{r+i(q-1)}$ for any $0\leq i\leq q$.
\end{itemize}
The number of pairs $(a,b)$ satisfying $c(a,b)=0$ is $q^3$, because they correspond to the affine points of $\He'$, and those satisfying $c\ne 0$ are $(q^4-q^3)$.
\end{proof}

\begin{corollary}\label{interettey} Let $\mathcal{L}$ be any horizontal line $\{y=b\}$,
with $b\in\FQ$. Then if $\Tra(b)=0$, $\mathcal{L}$ intersects $\He$ in one affine point, otherwise, if $\Tra(b)\ne 0$, $\mathcal{L}$ intersects $\He$ in $q+1$ affine points.
\end{corollary}
\begin{proof}
Apply Lemma \ref{intertot} with $a=0$.
\end{proof}

%==================================================================

\newpage
\section{Intersection between Hermitian curve $\He$ and parabolas}
\label{parabole}
We recall the definition of the Hermitian curve $\He$ on $\FQ$, i.e.
$$x^{q+1}=y^q+y.$$ 

Given two curves $X$ and $Y$ lying in the affine plane $(\Fq)^2$ it is
interesting to know the number of (affine plane) points
that lie in both curves, disregarding multiplicity and other similar notions.
We call this number \textit{their planar intersection}.
This knowledge may have applications for the codes constructed from  $X$ and $Y$.
As regards $\He$, it is interesting for coding theory applications
\cite{CGC-cod-art-couvreur2011dual,CGC-cod-art-ballico2012goppa,CGC-cod-art-ballico2012geometry,CGC-cod-art-fontanari2011geometry}
to consider an arbitrary parabola $y=ax^2+bx+c$ over
$\FQ$ and to compute their planar intersection.
Moreover, it is essential to know precisely the number of parabolas
having a given planar intersection with $\He$.
Only partial results were known \cite{CGC-alg-art-dondur10,CGC-alg-art-dondurkor09}, 
we present here for the first time a complete classification in the following theorem.\\

\begin{theorem}\label{teo.principe}
For $q$ odd, the only possible planar intersections of $\He$ and a parabola are $\{0,1,q-1,q,$ $q+1,2q-1,2q\}$.
For any possible mutual intersection we provide in the next tables the exact number of parabolas sharing that value.

{\small{\begin{table}[h]
\centering
\begin{tabular}{|l||c|c|c|}
\hline
$\#\He\cap$ parabola & 0 & 1 & $q-1$  \\
\hline
$\#$ parabolas & $q^2(q+1)\frac{(q-1)}{2}$ & $q^2(q+1)\frac{q(q-3)}{2}$ & $q^2(q+1)\frac{q(q-1)^2}{2}$  \\
\hline
\multicolumn{4}{c}{$ \quad$} \\
\hline
$\# \He\cap$ parabola & $q$ &\multicolumn{2}{c|}{ $q+1$} \\
\hline
$\#$ parabolas & $q^2(q+1)(q^2-q+1)$ &\multicolumn{2}{c|}{$q^2(q+1)\frac{q(q-1)(q-3)}{2}$}  \\
\hline
\multicolumn{4}{c}{$ \quad$} \\
\hline
$\# \He\cap$ parabola & $2q-1$ & \multicolumn{2}{c|}{ $2q$}\\
\hline
$\#$ parabolas & $q^2(q+1)\frac{q(q-1)}{2}$ & \multicolumn{2}{c|}{$q^2(q+1)\frac{(q-1)}{2}$} \\
\hline
\end{tabular}
\end{table}}}

For $q$ even, the only possible planar intersections  of $\He$ and a parabola  are $\{1,q-1,q+1,2q-1\}$.
For any possible mutual intersection we provide in the next tables the exact number of parabolas sharing that value.

\begin{table}[h]
\centering
\begin{tabular}{|l||c|c|}
\hline
$\# \He\cap$ parabola &  1 & $q-1$  \\
\hline
$\#$ parabolas & $q^3(q+1)(\frac{q}{2}-1)$ & $q^3(q+1)(q-1)\frac{q}{2}$\\
\hline
\multicolumn{3}{c}{$ \quad$} \\
\hline
$\# \He\cap$ parabola &  $q+1$ & $2q-1$ \\
\hline
$\#$ parabolas & $q^3(q+1)(q-1)(\frac{q}{2}-1)$ & $q^3(q+1)\frac{q}{2}$ \\
\hline
\end{tabular}
\end{table} 

\end{theorem}

%\newpage 

We begin with some simple lemmas.

\begin{lemma}\label{sol}
Let $t\in\FQ^*$, then there is a solution of $x^{q-1}=t$ if and only if $\No(t)=1$.
In this case, $x^{q-1}=t$ has exactly $(q-1)$ distinct solutions in $\FQ$.
\end{lemma} 
\begin{proof}
We consider the function $f:\FQ\rightarrow\FQ$ such that $f(x)=x^{q-1}$.
We want to prove that $t\in\I(f)\iff \No(t)=1$.\\

We can note that if $t=\alpha^{\lambda(q-1)}$ for any $\lambda$, then
$x=\alpha^{\lambda}$ is a solution of $x^{q-1}=t$. We claim that the solutions are:
$$x=\alpha^{\lambda+k(q+1)} \mbox{ con } 0\leq k\leq q-2.$$
In fact $(\alpha^{\lambda+k(q+1)})^{q-1}=\alpha^{\lambda(q-1)}=t$ per $0\leq k\leq q-2$.
So if the equation $x^{q-1}=t$ has at least one solution,
then it has at least $q-1$ distinct solutions. Since the equation degree is $q-1$, then it has exactly  $q-1$ distinct solutions.\\

Further we can note that if $t\in\I(f)\implies \No(t)=1$. In fact  
$\No(t)=t^{q+1}=(x^{q-1})^{q+1}=1$.

So $\I(f)\subset\{\No(t)=1\}$. But $|\I(f)|=q+1$ and $|\{\No(t)=1\}|=q+1$
hence  $\I(f)=\{\No(t)=1\}$.
\end{proof}

\begin{remark}
We note that $4\No(a)=\No(2a)$ for any $a\in\FQ$.
\end{remark}

\begin{lemma}\label{aquad}
If $q$ is odd and $4\No(a)=1$ then $a$ is a square in $\FQ$.
\end{lemma} 
\begin{proof}
Let $a=\alpha^k$, so that $4\No(a)=1\implies 4\alpha^{k(q+1)}=1\implies 4\beta^k=1$.
Since $4$ is a square in $\Fq$, we have 
$$4=\beta^{2t}\implies \beta^{2t+k}=1\implies 2t+k\equiv 0\mod q-1,$$ 
so that $k$ is even and $a$ is a square.
\end{proof}

We recall a well-known result in linear algebra:

\begin{lemma}\label{lembase}
Let $f:V\rightarrow W$ be a linear function and $f(\bar a)=a\in\I(f)$.
Then $f^{-1}(a)=\bar a + \ker(f)$ and
$$\dim\ker(f)+\dim\I(f)=\dim V$$
\end{lemma}

\begin{lemma}\label{lemlin}
For any $a\in\FQ$, let $f:\FQ\rightarrow\FQ$, $f(x)=2ax-x^q$. Then $f$ is $\Fq$-linear.
\end{lemma}
\begin{proof}$\,$\\
$\qquad \begin{array}{lrll}
  \forall\,c,d\in\FQ & f(c+d) & = & 2ac-c^q+2ad-d^q=f(c)+f(d).\\
  \forall\,c\in\FQ\quad\forall\,k\in\Fq\,\, & f(kc) & = & 2akc-k^qc^q=2akc-kc^q=kf(c).\\
\end{array}$\\
\end{proof}

Thanks to Lemma \ref{lemlin} and Lemma \ref{lembase} we have the following corollary:

\begin{corollary}\label{corFON}
Let $f:\FQ\rightarrow\FQ$ such that $f(x)=2ax-x^q$. Then the equation $f(x)=k$
has $q$ distinct solutions if $k\in\I(f)$, otherwise it has $0$ solutions.
\end{corollary}

\begin{lemma}\label{lemtraccia}
Let $y=ax^2+bx+\bar c$ and $y=ax^2+bx+c$ be two parabolas. If $\Tra(\bar c)=\Tra(c)$,
then the planar intersections between the Hermitian curve $\He$ and the parabolas are the same.
\end{lemma}
\begin{proof}
From $\{y=ax^2+bx+\bar c\}\cap\He$ and $\{y=ax^2+bx+c\}\cap\He$, we obtain by direct substitution respectively
$x^{q+1}=a^qx^{2q}+ax^2+b^qx^q+bx+\Tra(\bar c)$ and $x^{q+1}=a^qx^{2q}+ax^2+b^qx^q+bx+\Tra(c)$.
If $\Tra(\bar c)=\Tra(c)$, the two equations are identical.
\end{proof}

We define the function $F_a:\FQ\rightarrow\Fq$ such that
\begin{equation}\label{prin}
  F_a(x)=\No(x)-\Tra(ax^2).
\end{equation} 
We can note that the following propriety holds for the function $F_a$:

\begin{lemma}\label{lemmaFa}
If $\omega\in\Fq$, then $F_a(\omega x)=\omega^2 F_a(x)$.
\end{lemma}
\begin{proof}
Since $\omega\in\Fq$, we have $F_a(\omega x)=\No(\omega x)-\Tra(a(\omega x)^2)=
\omega^{q+1}x^{q+1}-a^q\omega^{2q}x^{2q}-a\omega^2 x^2=
\omega^2(x^{q+1}-a^q x^{2q}-ax^2)=\omega^2 F_a(x).$
\end{proof}

%%%%%%%%%%%%%%%%%%%%%%%%%%%%%%%%%%%%%%%%%%%%%%%%%%%%%%%%%%%%%%%%%%%%%%%%%%%
%%%%%%%%%%%%%%%%%%%%%%%%%%%%%%%%%%%%%%%%%%%%%%%%%%%%%%%%%%%%%%%%%%%%%%%%%%%
%%%%%                    AUTOMORFISMO                                %%%%%%
%%%%%%%%%%%%%%%%%%%%%%%%%%%%%%%%%%%%%%%%%%%%%%%%%%%%%%%%%%%%%%%%%%%%%%%%%%%
%%%%%%%%%%%%%%%%%%%%%%%%%%%%%%%%%%%%%%%%%%%%%%%%%%%%%%%%%%%%%%%%%%%%%%%%%%%

We consider the automorphism group $Aut(\He)$ of the Hermitian curve. Any
automorphism $\sigma\in Aut(\He)$ has the following form, as in \cite{CGC-alg-art-xing95}:
$$\sigma\left(\begin{array}{c}x \\ y\end{array}\right)=
\left(\begin{array}{c}\epsilon x+\gamma \\ \epsilon^{q+1}y+\epsilon\gamma^qx+\delta\end{array}\right)$$
with $(\gamma,\delta)\in\He$, $\epsilon\in\FQ^*$. If we choose $\epsilon=1$ we obtain the following automorphisms 

\begin{equation}\label{autom}
  \left\{\begin{array}{l}
  x\longmapsto x+\gamma\\
  y\longmapsto y+\gamma^qx+\delta
\end{array}\right.\qquad\mbox{ with }(\gamma,\delta)\in\He,
\end{equation}
that form the subgroup $\Gamma$ with $q^3$ elements, see Section\texttt{ II} of \cite{CGC-cod-art-stichtenoth1988note}.\\
The reason why we are interested in the curve automorphisms is the following.
If we apply any $\sigma$ to any curve $\chi$ in the affine plane, then the planar intersections  between
$\sigma(\chi)$ and $\He$ will be the same as the planar intersections between $\chi$ and $\He$.
So, if we find out the number of intersections between $\chi$ and $\He$, we will
automatically have the number of intersection between $\sigma(\chi)$ and $\He$ for all $\sigma\in Aut(\He)$.
This is convenient because we can isolate special classes of parabolas
that act as representatives in the orbit $\{\sigma(\chi)\}_{\sigma\in\Gamma}$.
These special types of parabolas may be easier to handle.\\

Two special cases of interest are parabolas
of the form $y=ax^2$ and $y=ax^2+c$.\\

If we apply (\ref{autom}) to $y=ax^2$, we obtain
\begin{eqnarray}\label{par}
  y=ax^2+x(2a\gamma-\gamma^q)+a\gamma^2-\delta.
\end{eqnarray}
 
while if we apply (\ref{autom}) to $y=ax^2+c$ we obtain
\begin{eqnarray}\label{parc}
  y=ax^2+x(2a\gamma-\gamma^q)+a\gamma^2-\delta+c.
\end{eqnarray}

In the general case, if we have $y=ax^2+bx+c$ and apply the automorphism (\ref{autom}) we obtain 
\begin{equation}\label{autpartot}
  y=ax^2+(2a\gamma-\gamma^q+b)x+a\gamma^2+b\gamma-\delta+c.
\end{equation}

We need to study two distinct cases depending on the field characteristic.

Subsection \ref{odd} is devoted to the proof of the odd case of Theorem \ref{teo.principe},
while Subsection \ref{even} is devoted to the proof of the even case of Theorem \ref{teo.principe}.

%%%%%%%%%%%%%%%%%%%%%%%%%%%%%%%%%%%%%%%%%%%%%%%%%%%%%%%%%%%%%%%%%%%%%%%%%%%
%%%%%%%%%%%%%%%%%%%%%%%%%%%%%%%%%%%%%%%%%%%%%%%%%%%%%%%%%%%%%%%%%%%%%%%%%%%
%%%%%%%%%%%%%%%%%%%%%%%%%%%%%%%%%%%%%%%%%%%%%%%%%%%%%%%%%%%%%%%%%%%%%%%%%%%
%%%%%%%%%%%%%%%%%%%%%%%%%%%%%%%%%%%%%%%%%%%%%%%%%%%%%%%%%%%%%%%%%%%%%%%%%%%

\subsection{Odd characteristic}\label{odd}

In this subsection, $q$ is always odd.

First, we consider a parabola of the form $y=ax^2+ c$. Intersecting it with
the Hermitian curve, we obtain $x^{q+1}=a^q x^{2q}+ax^2+\Tra(c)$ which is equivalent to
\begin{equation}\label{eqiniziale}
  \No(x)-\Tra(ax^2)=F_a(x)=\Tra(c)
\end{equation}
Now we have to study the number of solutions of (\ref{eqiniziale}).
From this equation we get $a^q x^{2q}-x^{q+1}+ax^2=-\Tra(c)$, that is,
\begin{equation}\label{eqquad}
x^2(a^q x^{2q-2}-x^{q-1}+a)=-\Tra(c).
\end{equation}
Now we set $x^{q-1}=t$ and we need to factorize the polynomial $a^q t^2-t+a$ in $\FQ[t]$.
We obtain
$$t_{1,2}=\frac{1\pm\sqrt{1-4\No(a)}}{2a^q}=\frac{1\pm\sqrt{\Delta}}{2a^q}$$
where we set $\Delta=1-4\No(a)$, that is,
\begin{equation}\label{eqini}
a^qx^2\left(x^{q-1}-\frac{1+\sqrt{\Delta}}{2a^q}\right)
\left(x^{q-1}-\frac{1-\sqrt{\Delta}}{2a^q}\right)=-\Tra(c).    
\end{equation}
Since $\Delta\in\Fq$, there is $z\in\FQ$ such that $\Delta=z^2$, and so the polynomial is in $\FQ[x]$.\\ 
We note that $\Delta=0\iff \No(2a)=1$. So, in this special case, (\ref{eqini}) becomes 
$a^qx^2(x^{q-1}-2a)^2=-\Tra(c)$. We have proved the following lemma:
\begin{lemma}\label{deltaquad}
By intersecting a parabola $y=ax^2+c$ with $\No(2a)=1$ and the Hermitian curve,
we obtain the following equation
$$a^qx^2(x^{q-1}-2a)^2=-\Tra(c).$$
\end{lemma}
Recall that $\alpha$ is a primitive element of $\FQ$ and $\beta=\alpha^{q+1} $ is a primitive element of $\Fq$.
\begin{lemma}\label{lem:ab}
Let $x=\alpha^j\beta^i$, with $j=0,\ldots,q$ and $i=0,\ldots,q-2$. Then
\begin{itemize}
  \item If $4\No(a)\ne 1$, then all non-zero values $F_a(\alpha^j\beta^i)$ are  the elements of $\Fq^*$.
  \item If $4\No(a)=1$, then all non-zero values $F_a(\alpha^j \beta^i)$ are half of the elements of $\Fq^*$.
\end{itemize}
\end{lemma}
\begin{proof}
We fix an index $j$ such that $F_a(\alpha^j)\ne 0$.
The set of the values $\{F_a(\alpha^j \beta^i)\}_{0\leq i\leq q-2}=
\{\beta^{2i}F_a(\alpha^j)\}_{0\leq i \leq q-2}$ contains half of the elements
of $\Fq^*$, since $q$ is odd (and $\frac{q-1}{2}$ is an integer) and so
$\beta^{2(\frac{q-1}{2})}=\alpha^{q^2-1}=1$. 

If $4\No(a)\ne 1$, by varying $j$, we can obtain every element of $\Fq^*$.

In fact, $F_a(x)=-x^2(a^q x^{2q-2}+a- x^{q-1})$, so  
$$
F_a(\alpha^j)=\No(\alpha^j)-\Tra(a(\alpha^j)^2)=\beta^j-
a\alpha^{2j}-a^q\alpha^{2jq}=\beta^j-\beta^{r_j},
$$
where $0\leq r_j\leq q^2-2$ and $\beta^{2i}F_a(\alpha^j)=\beta^{2i+j}-\beta^{2i+r_j}$, that are all elements of $\Fq^*$.

By Lemma \ref{deltaquad}, if $4\No(a)=1$ then $F_a(x)$ becomes $-a^qx^2(x^{q-1}-2a)^2$,
so $\beta^{2i}F_a(\alpha^j)=-a^q\beta^{2i}(\alpha^{jq}-2a\alpha^j)^2$,
that are half of the elements of $\Fq^*$.
\end{proof}

Now we study the number of solutions of equation (\ref{eqiniziale}).\\
We consider two cases: when $\Tra(c)=0$ and when $\Tra(c)\ne 0$.
\begin{itemize}
  \item[$*$] Case $\Tra(c)=0$. By Lemma \ref{lemtraccia}, it is enough to study
the case $c=0$, which is the intersection between $\He$ and $y=ax^2$.
By (\ref{eqini}) we have
$$a^qx^2\left(x^{q-1}-\frac{1+\sqrt{\Delta}}{2a^q}\right)
\left(x^{q-1}-\frac{1-\sqrt{\Delta}}{2a^q}\right)=0.$$
We must differentiate our argument depending on $\Delta$. Recall that $\Delta\in\Fq$.
\begin{itemize}
    \item[-] $\mathbf{\Delta}=0$.
By Lemma \ref{deltaquad}, (\ref{eqini}) becomes
$$
a^qx^2(x^{q-1}-2a)^2=0.
$$
So we have always one solution $x=0$ and the solutions of {\small{$\{x^{q-1}=2a\}$}}, which are $q-1$,
since $\No(2a)=1$ and so we can apply Lemma \ref{sol}.
Therefore, in this case, we have $q$ points of intersections between the parabola and the Hermitian curve $\He$. 

\noindent Since $\No(2a)=1$, we have $(q+1)$ distinct $a$'s.
    \item[-]$\mathbf{\Delta}=1$. That is, $\No(2a)=0\iff a=0$, but this is impossible.
    \item[-]$\mathbf{\Delta}\in\Fq\backslash\{0,1\}$.
We note that any element in $\Fq$ can always be written as $z^2$ with $z^2\in\FQ$.
In order to study the solutions of (\ref{eqini}), we can study the solutions of the following equations
\begin{equation}\label{stella}
    x^{q-1}=\frac{1\pm z}{2a^q}.
\end{equation}
By Lemma \ref{sol} we know that $x^{q-1}=\frac{1+z}{2a^q}$ have solutions
if and only if $\No\left(\frac{1+z}{2a^q}\right)=1$. Note that
{\small{$$\No\left(\frac{1+z}{2a^q}\right)=1\iff\frac{1-z^2}{(1-z)^{q+1}}=1\iff 1+z=(1-z)^q\iff -z=z^q$$}}

\noindent We obtain the same result for $x^{q-1}=\frac{1-z}{2a^q}$.\\
If $z$ is a solution of (\ref{stella}) and $z\in\Fq$, then $z$ satisfy simultaneously $z^q=z$ and $z^q=-z$. Since $q$ is odd, this is possible only when $z=0$, which implies $\Delta=0$, that is not admissible.\\

\noindent  Returning to count the intersection points, thanks to the previous discussion of the solution of (\ref{stella}), we have to consider two distinct cases:
\begin{enumerate}
    \item $z=z^q$, that is,  $z\in\Fq$. It is easy to prove that there
    are $\frac{q-1}{2}-1=\frac{q-3}{2}$ possible values of $z^2$, since $z\ne 0,1$.
    So we have $(q+1)\frac{q-3}{2}$ values of $a$ such that we have
    only $one$ point of intersection (with $x=0$).
    \item $z=-z^q$. The equation $-z=z^q$ has only one solution in $\Fq$, so the other $q-1$ solutions are in $\FQ\backslash\Fq$.
    For such $z$, we have $2(q-1)+1=2q-1$ points of intersection.
    That is, $q-1$ solutions from equation $x^{q-1}=\frac{1-z}{2a^q}$, $q-1$ solutions from equation $x^{q-1}=\frac{1+z}{2a^q}$ and one point from $x=0$.\\
It is simple to verify that the $z^2$ such that $z\in\FQ\backslash\Fq$ are $\frac{q-1}{2}$.
So we have $(q+1)\frac{q-1}{2}$ values of $a$ such that we have exactly
$2q-1$ points of intersection.

\end{enumerate}
\end{itemize}

Now we apply the automorphism and we want to calculate how many different parabolas
we can obtain. Applying (\ref{autom}) to $y=ax^2$ we obtain (\ref{par}).
For the moment, we restrict our counting argument to the case $\Delta\ne 0$.
We can note that if $\Delta\ne 0$, we have a maximal orbit, that is,
all possible parabolas are distinct (they are $q^3$ because $\Gamma$ has $q^3$ elements).
In other words, we claim that it is impossible that we obtain
two equal parabolas with $(\gamma,\delta)\ne(\bar\gamma,\bar\delta)$.
To prove that, we have to solve the following system:
{\footnotesize{$$
\left\{\begin{array}{l}
2a\bar\gamma-\bar\gamma^q=2a\gamma-\gamma^q\\
a\bar\gamma^2-\bar\delta=a\gamma^2-\delta\\
\gamma^{q+1}=\delta^q+\delta\\
\bar\gamma^{q+1}=\bar\delta^q+\bar\delta\\
1-4a^{q+1}\ne 0.
\end{array}\right.
$$}}

\noindent However, $2a\bar\gamma-\bar\gamma^q=2a\gamma-\gamma^q\iff
2a(\bar\gamma-\gamma)=\bar\gamma^q-\gamma^q=(\bar\gamma-\gamma)^q\implies
4a^{q+1}(\bar\gamma-\gamma)^q(\bar\gamma-\gamma)=(\bar\gamma-\gamma)(\bar\gamma-\gamma)^q\iff
4a^{q+1}=1$. And it is impossible, because $\Delta\ne 0$.\\
So when $\Delta\ne 0$ we have exactly $q^3$ distinct parabolas that have
the same planar intersections  with $\He$ as $y=ax^2$ has.

  \item[$*$] Case $y=ax^2+c$, with $\Tra(c)\ne 0$.
As in previous case, we have to differentiate depending on $\Delta$.
\begin{itemize}
    \item[-] If $\Delta=z^2$ and $z\in\Fq$, we know that $F_a(x)$ vanishes only if $x=0$.
If $x\ne 0$, then by Lemma \ref{lem:ab}, $F_a(\beta^i\alpha^j)=\beta^{2i}F_a(\alpha^j)$
can assume every value of $\Fq^*$. But $j$ can assume $q+1$ distinct values,
so every value of $\Fq^*$ is obtained $q+1$ times. So the equation $F_a(x)=\Tra(c)$
has exactly $q+1$ solutions.
    \item[-] If $\Delta=z^2$ and $z\in\FQ\backslash\Fq$, we know that $F_a(x)=0$
has $2q-1$ solutions. So there are exactly two distinct values of $j$ such that
$F_a(\alpha^j)=0$, and so every value in $\Fq^*$ is obtained $q-1$ times. So
the equation $F_a(x)=\Tra(c)$ has exactly $q-1$ solutions.
    \item[-] If $\Delta=0$ we have $4a^{q+1}=1$.\\ 
So (\ref{prin}) can be written as $ a^qx^2(x^{q-1}-2a)^2=-\Tra(c)$, that is,
\hspace{-1.5cm}\begin{equation}\label{dzero}
  x^2(x^{q-1}-2a)^2=-4a\Tra(c)=-4a\beta^r
\end{equation}
for some fixed $r$ with $1\leq r\leq q-1$.\\ 
Note that (\ref{dzero}) can be written as $f(x)^2=-4a\Tra(c)$,
where $f$ is as in Lemma \ref{lemlin}, that is, $f(x)=x^{q}-2ax$.\\
We note that $-4a\beta^r$ is always a square in $\FQ$.
In fact $-4\beta^r$ is a square because lies in $\Fq$,
and also $a$ is a square by Lemma~\ref{aquad}. Let us write $-4a\beta^r=\alpha^{2h}$, so (\ref{dzero}) becomes
$x(x^{q-1}-2a)=\pm\,\alpha^{\,h}$. Clearly, $0\leq h\leq \frac{q^2-1}{2}$.

\noindent We consider the ``positive'' case:
\begin{equation}\label{cond1}
f(x)=x^{q}-2ax=\alpha^{\,h}.
\end{equation} 
\noindent It is simple to prove that if $x$ is a solution of equation
$x^q-2ax=\alpha^h$, then $-x$ is a solution of the equation $x^q-2ax=-\alpha^h$.
So by Corollary~\ref{corFON} the equation $F_a(x)=\Tra(c)$ has $0$ solutions
if $\alpha^h$ is not in $\I(f)$ or $2q$ solution if $\alpha^h$ is in $\I(f)$.

\end{itemize}
\end{itemize}

We consider a parabola $y=ax^2+bx+c$ and apply the automorphism (\ref{autom})
and we obtain (\ref{autpartot}).\\
Note that, for any $k\in\FQ$,
\begin{equation}\label{bdiv}
  2a\gamma-\gamma^q+b=k\implies 2ab^q+b=2ak^q+k.
\end{equation}
In fact $b^q=(k-2a\gamma+\gamma^q)^q=k^q-\frac{1}{2a}\gamma^q+\gamma=k^q+
\frac{1}{2a}(-\gamma^q+2a\gamma)=k^q+\frac{1}{2a}(k-b)$.\\

A consequence is that $2a\gamma-\gamma^q+b=0\implies 2ab^q+b=0$.\\

We consider two distinct cases $2a\gamma-\gamma^q+b=0$ and $2a\gamma-\gamma^q+b\ne 0$.\\

%%%%%%%%%%%%%%%%%%%%%%%%%%%%%%%%%%%%%%%%%%%%%%%%%%%%%%%%%%%%%%%%%%%
%%%%%%%%%%%%%%%%%%%%%%%%%%%%%%%%%%%%%%%%%%%%%%%%%%%%%%%%%%%%%%%%%%%

\noindent\begin{tabular}{|c|}
\hline
$2a\gamma-\gamma^q+b\ne 0$.\\
\hline
\end{tabular}

\begin{thm}
Let $y=ax^2+bx+c$ be a parabola with $2ab^q+b\ne 0$ and $\No(2a)=1$.
Then there exists a $\gamma$ such that for any $\delta$ if we apply the automorphism
(\ref{autom}) we obtain $y=ax^2+(2a\gamma-\gamma^q+b)x+a\gamma^2+ b\gamma-\delta+c$,
with $2a\gamma-\gamma^q+b\ne 0$. We can write any such parabola as $y=(ux+uv)^2$ where $a=u^2$ and $v^q+2av\ne 0$.
\end{thm}  
\begin{proof}
Thanks to (\ref{bdiv}) with $k\ne 0$ we have that, since $2ab^q+b\ne 0$,
then $\exists\,\gamma$ such that $2a\gamma-\gamma^q+b\ne 0$.\\
Let $k\in\FQ$ such that $2a\gamma-\gamma^q+ b=k\ne 0$. Using Corollary
\ref{corFON}, if there exists at least one solution of $2a\gamma-\gamma^q=k-b$,
then there exists $q$ solutions. So we have at least $q$ different $\gamma$'s
that verify the previous equation.\\

To prove that any parabola as in (\ref{autpartot}) can be written as $y=(ux+uv)^2$
with $a=u^2$ and $v^q+2av\ne 0$, it is sufficient to prove that the solutions of the following system contain all $c$'s.
$$
\left\{\begin{array}{l}
2a\gamma-\gamma^q+b=2av\ne 0\\
a\gamma^2+b\gamma-\delta+c=av^2\ne 0\\
\gamma^{q+1}=\delta^q+\delta\\
1-4a^{q+1}=0
\end{array}\right.
$$
Using (\ref{bdiv}) the first equation of system $2a\gamma-\gamma^q+b=2av$
implies that $v^q+2av\ne 0$. In fact if we consider (\ref{bdiv}) with $k=2av$,
we have $0\ne 2ab^q+b=2ak^q+k=2a(2av)^q+2av=v^q+2av$.\\
By the second equation we have $c=\delta+av^2-a\gamma^2-b\gamma$.
So for any $\gamma$ (and there are $q$ possible $\gamma$'s), there are $q$
distinct $\delta$'s (by the curve equation). So we have $q^2$ different $c$'s,
that is, all possible $c$'s.\\
Finally, we can write (\ref{autpartot}) as $y=a(x+v)^2$.
By Lemma \ref{aquad}, $a=u^2$ is a square so $y=(ux+uv)^2$. 
\end{proof}

\begin{thm}
Let $a,v\in\FQ$ such that $\No(2a)=1$ and $v^q+2av\ne 0$. Then the Hermitian curve $\He$
intersects a parabola $y=a(x+v)^2$ in $q$ points.
\end{thm}
\begin{proof}
We have to solve the system
$$
\left\{\begin{array}{l}
y=(ux+uv)^2\\
x^{q+1}=y^q+y
\end{array}\right.\implies x^{q+1}=(ux+uv)^{2q}+(u+uv)^2
$$
By a change of variables $z=ux+uv$, we obtain $(\frac{z-uv}{u})^{q+1}=z^{2q}+z^2$, so we have
$$
-(uv)z^q-(uv)^{q}z+(uv)^{q+1}=u^{q+1}z^{2q}+u^{q+1}z^2-z^{q+1}=
u^{q+1}(z^q-2u^{q+1}z)^2.
$$
Since $\No(2a)=1$ and $a=u^2$, we have $u^{q+1}=\pm \frac{1}{2}$ we have
\begin{align}
  \frac{1}{2}(z^q-z)^2=\No(uv)-\Tra(z(uv)^q)\label{eqz}\\
  -\frac{1}{2}(z^q+z)^2=\No(uv)-\Tra(z(uv)^q)\label{eqzm}
\end{align}

We consider two cases:\\
\begin{itemize}
  \item[$*$] If $u^{q+1}=\frac{1}{2}$, we can note that $(z^q-z)^2$ is not a square in $\Fq$ if $z^q-z\ne 0$.\\ 
In fact, if we suppose by contradiction that $(z^q-z)^2=\beta^{2r}$,
then $z^q-z=\beta^{r}\in \Fq$ but also $z^q+z\in\Fq$, so $-2z\in\Fq\iff z\in\Fq$
and so $z^q-z=0$, which is a contradiction.
  \item[$*$] If $u^{q+1}=-\frac{1}{2}$, we can note that $(z^q+z)^2$ is a square in $\Fq$, because $z^q+z\in\Fq$.
\end{itemize}
Let $t=\No(uv)-\Tra(z(uv)^q)$. So $t\in \Fq$.
Thanks to (\ref{eqz}) we have $2t=(z^q-z)^2$, while (\ref{eqzm}) becomes $-2t=(z^q+z)^2$.\\

So we have $\frac{q-1}{2}$ values of $t$ (that are all the non-squares)
and $t=0$ when $u^{q+1}=\frac{1}{2}$, and $\frac{q-1}{2}$ values of $t$
(that are all the squares) and $t=0$ when $u^{q+1}=-\frac{1}{2}$.\\
\indent Now we consider separately the case $t=0$ and $t\ne 0$.

\begin{itemize}
  \item[-] We claim that if $t=0\mbox{ and }u^{q+1}=\pm \frac{1}{2}\implies z\in\Fq$.
  We show only the case $u^{q+1}=\frac{1}{2}$. With these assumptions (\ref{eqz}) becomes
$$
-(uv)z^q-(uv)^{q}z+(uv)^{q+1}=0\iff z=\frac{(uv)^{q+1}}{(uv)^{q}+uv}.
$$
We can note that since $v^q+2av\ne 0$, then $(uv)^{q}+uv\ne 0$.
In fact suppose that $v^q+2av=0$, so $(uv)^{q}+uv=-\frac{1}{2u}2av+uv=0$.\\
We have to verify that $z^q=z$. Indeed $z^q=\frac{(uv)^{q+1}}{uv+(uv)^{q}}=z$.\\
Similar calculations (here omitted) show the case $u=-\frac{1}{2}$.
  \item[-] We claim that if $t\ne 0\mbox{ and } u^{q+1}=\pm \frac{1}{2}\implies z\not\in\Fq$.
  With these assumptions, we show only the case $u^{q+1}=\frac{1}{2}$. We have $(z^q-z)^2=2t=\alpha^{2r}$,
  that is, $z^q=z \pm \alpha^r$. Now we substitute $z^q$ in
$-(uv)z^q-(uv)^{q}z+(uv)^{q+1}=t $ and we obtain
$-(uv)(\pm\alpha^r+z)-(uv)^{q}z+(uv)^{q+1}=\frac{1}{2}\alpha^{2r}$, that is,
\begin{equation}\label{solz}
z=\frac{(uv)^{q+1}-\frac{1}{2}\alpha^{2r}\mp uv\alpha^r}{\Tra(uv)}
\end{equation}
We can note that $\alpha^{qr}=-\alpha^r$, in fact $2t=\alpha^{2r}\in\Fq$,
so $(\alpha^{2r})^q=\alpha^{2r}$, that is, $\alpha^{rq}=\pm\alpha^{r}$
but $\alpha^r\not\in\Fq$ (since $2t$ is not a square in $\Fq$) so $\alpha^{qr}=-\alpha^r$. We have thus proved
$$
z=\frac{(uv)^{q+1}-\frac{1}{2}\alpha^{2r}\mp uv\alpha^r}{\Tra(uv)}\mbox{ and }
z^q=\frac{(uv)^{q+1}-\frac{1}{2}\alpha^{2r}\pm (uv)^q \alpha^r}{\Tra(uv)}
$$

\noindent Now we have to verify that the two $z$'s as in (\ref{solz}) are solutions of (\ref{eqz}).
We have $z^q-z=\pm\alpha^r$ and $\No(uv)-\Tra(z(uv)^q)=t$. So
$$\pm\alpha^r=z^q-z\iff\pm \Tra(uv)\alpha^r=\pm(uv)^{q}\alpha^{r}\pm uv\alpha^r$$
and
{\small{
$$
\begin{array}{cl}
 & (uv)^{q+1}-z(uv)^q-z^q(uv)=t\\
\iff & (uv)^{q+1}\Tra(uv)-(uv)^q((uv)^{q+1}-\frac{1}{2}\alpha^{2r}\pm uv\alpha^r)+\\
&-uv((uv)^{q+1}-\frac{1}{2}\alpha^{2r}\pm uv \alpha^r)= \Tra(uv) t\\
\iff & (uv)^{q+1}\Tra(uv)+t \Tra(uv)-(uv)^{2q+1}-(uv)^{q+2}= \Tra(uv) t\\
\iff & (uv)^{q+1}\Tra(uv)-(uv)^{2q+1}-(uv)^{q+2}= 0.\\
\end{array}
$$}}
So the $z$'s are solutions of (\ref{eqz}).\\
Similar calculations (omitted here)  show the case $u^{q+1}=-\frac{1}{2}$.
\end{itemize}
Therefore, we have two solutions for any $t$ not a square in $\Fq^*$
and we have only one solution when $t=0$.
That is, we get a total of $\frac{q-1}{2}2+1=q$ intersections.

The same holds for the case with $u^{q+1}=-\frac{1}{2}$.
\end{proof}

%%%%%%%%%%%%%%%%%%%%%%%%%%%%%%%%%%%%%%%%%%%%%%%%%%%%%%%%%%%%%%%%%%%
%%%%%%%%%%%%%%%%%%%%%%%%%%%%%%%%%%%%%%%%%%%%%%%%%%%%%%%%%%%%%%%%%%%
\break

Now we consider the second case.\\

\noindent\begin{tabular}{|c|}
\hline
$2a\gamma-\gamma^q+b=0$.\\
\hline
\end{tabular}
$\,$\\

We note that if $2a\gamma-\gamma^q+b=0$ then $2ab^q+b=0$, and so (\ref{autpartot})
is actually $y=ax^2+c$. Now we apply the automorphism (\ref{autom})
to the parabola $y=ax^2+c$ and we obtain (\ref{parc}). Now if

\begin{itemize}
  \item[-] $\Delta\ne 0$.\\
The parabolas in (\ref{par}) are all distinct.

The values of $c$ such that $\Tra(c)\ne 0$ are exactly $q^2-q$,
but we must be careful and not count twice the same parabola.
In particular, if two parabolas share $a$ and $b$, then they are in the
same orbit if $\Tra(c)=\Tra(c')$.
So we must consider only one of these for any non-zero trace $\Tra(c)$. These values are $q-1$. 
Summarizing:
\begin{itemize}
    \item[$*$] If $\Delta=z^2$ and $z\in\Fq$ (and $\Tra(c)\ne 0$), then the number of parabolas with $q+1$ intersection is 
$$
\underbrace{(q+1)\frac{q-3}{2}}_{a}\underbrace{q^3(q-1)}_{b,c}=\frac{1}{2}
q^3(q^2-1)(q-3).
$$
    \item[$*$] If $\Delta=z^2$ and $z^q+z=0$ (and $\Tra(c)\ne 0$), then the number of parabolas with $q-1$ intersection is 
$$
\underbrace{(q+1)\frac{q-1}{2}}_{a}\underbrace{q^3(q-1)}_{b,c}=
\frac{1}{2}q^3(q+1)(q-1)^2.
$$
\end{itemize}
  \item[-] $\Delta=0$, that is, $4a^{q+1}=1$.

\noindent We want to understand how many different parabolas of the type\\ $y=ax^2+bx+\bar c$
(with $a$ fixed) we can obtain. So we have to study the number of pairs $(b,\bar c)$.

We note that  
\begin{equation}\label{tracce}
  \Tra(\bar c)=a^qb^2+\Tra(c).
\end{equation}
In fact
$$
\begin{array}{lll}
\Tra(\bar c) & = & (a\gamma^2-\delta)^q+a\gamma^2-\delta+\Tra(c)\\
& = & (a\gamma^2)^q+a\gamma^2-\gamma^{q+1}+\Tra(c)=a^q\gamma^2(\gamma^{q-1}-2a)^2+\Tra(c).
\end{array} 
$$
Let $\Tra(c)=k$, with $k\in\Fq^*$. Let us consider two distinct cases:
\begin{itemize}
    \item[-] $\Tra(c)=\Tra(\bar c)$. By (\ref{tracce}) we have that $\Tra(c)=\Tra(\bar c) \iff b=0$.\\ 
So the pairs $(0,\bar c)$ are exactly $q^2-q$, because they correspond to all $\bar c\in\FQ$ such that $\Tra(\bar c)\ne 0$.
    \item[-] $\Tra(c)\ne \Tra(\bar c)$. Then $\Tra(\bar c)=a^qb^2+k$.\\
Since $b=2a\gamma-\gamma^q$, then, by considering all possible $\gamma$'s, we obtain $q-1$ distinct $b$'s.\\
In fact, we can consider the function $f:\FQ\rightarrow \FQ$ such that
$f(\gamma)=2a\gamma-\gamma^q$. By Corollary \ref{corFON}, for any $t\in \I(f)$,
the equation $f(\gamma)=t$ has $q$ distinct solutions.\\
Since we are interested in the case $b\ne 0$, we have $\frac{(q^2-q)}{q}=q-1$ different $b$'s.
\noindent We can note that if $b$ is a solution of the equation $2a\gamma-\gamma^q=0$, also $-b$ is.\\
\noindent Since we are interested in the pairs $(b^2,\bar c)$,
we note that we have to consider the equation $\Tra(\bar c)=a^qb^2+k$,
so the pairs $(b^2,\bar c)$ are exactly $\frac{q-1}{2}(q^2-q)$.
In fact there are exactly $\frac{q-1}{2}$ distinct $b^2$ and for any pairs $(b^2, k)$
we have exactly $q$ distinct $\bar c\,$'s. While the possible $k$'s are exactly $q-1$
(because $\Tra(c)\ne 0$).\\
All possible pairs $(b,\bar c)$ are $2\frac{q-1}{2}(q^2-q)=(q-1)(q^2-q)$.
\end{itemize}

\noindent We fix $a$ and we obtain exactly
$(q-1)(q^2-q)+q^2-q=q^2(q-1)$ parabolas of the type $y=ax^2+bx+\bar c$.

In conclusion if $\Delta=0$ and $\Tra(c)\ne 0$, then we have
$q^2(q+1)\frac{q-1}{2}$ parabolas with $2q$ or $0$ intersections.
\end{itemize}
The last type of parabolas cannot be easily counted and so we obtain their number by difference.
\begin{claim} The parabolas that have $q$ intersections with the Hermitian curve $\He$ are $q^2(q+1)(q^2-q+1)$.
\end{claim}
\begin{proof}
The number of total parabolas is $q^4(q^2-1)$. By summing all parabolas that we already counted we obtain
$$q^2(q+1)\left(2\frac{q-1}{2}+q\frac{q-1}{2}(q-1+q-3)+\frac{q}{2}(q-1+q-3)+q^2-q+1\right)=$$
$$=q^2(q+1)(q^2(q-1))=q^4(q^2-1).$$
But this is exactly the number of parabolas, so we counted them all.
\end{proof}

We have proved the following theorems, depending on the condition\\ $\Tra(c)=0$ or $\Tra(c) \ne 0$.
\begin{theorem}\label{teo.intTr0}
Let $q$ be odd. The Hermitian curve $\He$ intersects a parabola $y=ax^2+c$ with $\Tra(c)=0$ in $2q-1, q$ or $1$ points.\\ 
Moreover, we have
\begin{description}
  \item[]$(q+1)\frac{q-1}{2}q^3$ parabolas that intersect $\He$ in $2q-1$ points.
  \item[]$q^2(q+1)(q^2-q+1)$ parabolas that intersect $\He$ in $q$ points.
  \item[]$(q+1)\frac{q-3}{2}q^3$ parabolas that intersect $\He$ in $one$ point.
\end{description}
\end{theorem}

\begin{theorem}\label{teo.intTrn0}
Let $q$ be odd. The Hermitian curve $\He$ intersects a parabola $y=ax^2+c$ with $\Tra(c)\ne 0$
in $2q, q+1,q-1$ or $0$ points.\\ 
Moreover, we have
\begin{description}
  \item[]$q^2(q+1)\frac{q-1}{2}$ parabolas that intersect $\He$ in $2q$ points.
  \item[]$q^3(q+1)(q-1)\frac{q-3}{2}$ parabolas that intersect $\He$ in $q+1$ points.
  \item[]$q^3(q+1)\frac{(q-1)^2}{2}$ parabolas that intersect $\He$ in $q-1$ points.
  \item[]$q^2(q+1)\frac{q-1}{2}$ parabolas that intersect $\He$ in $0$ point.
\end{description}
\end{theorem}
Therefore, thanks to Theorem \ref{teo.intTr0} and to Theorem \ref{teo.intTrn0}, 
we obtain the first half of Theorem \ref{teo.principe}.

%%%%%%%%%%%%%%%%%%%%%%%%%%%%%%%%%%%%%%%%%%%%%%%%%%%%%%%%%%%%%%%%%%%%%%%%%%
%%%%%%%%%%%%%%%%%%%%%%%%%%%%%%%%%%%%%%%%%%%%%%%%%%%%%%%%%%%%%%%%%%%%%%%%%%
%%%%%%%%%%%%%%%%%%%%%%%%%%%%%%%%%%%%%%%%%%%%%%%%%%%%%%%%%%%%%%%%%%%%%%%%%%
%%%%%%%%%%%%%%%%%%%%%%%%%%%%%%%%%%%%%%%%%%%%%%%%%%%%%%%%%%%%%%%%%%%%%%%%%%

\subsection{Even characteristic}\label{even}

In this subsection, $q$ is always even.

In this case we claim that it is enough to consider just two special cases:
$y=ax^2$ and $y=ax^2+c$. Before studing the two cases, we consider the following lemma:
\begin{lemma}\label{lem:abeven}
Let $x=\alpha^j\beta^i$, with $j=0,\ldots,q$ and $i=0,\ldots,q-2$; then the
values $F_a(\alpha^j\beta^i)$ that are not zero are all the elements of $\Fq^*$.
\end{lemma}
\begin{proof}
Fixing a index $j$, we have $F_a(\alpha^j \beta^i)=\beta^{2i}F_a(\alpha^j)$;
if $F_a(\alpha^j)=0$ we have finished, otherwise (by Lemma \ref{lemmaFa})
$\beta^{2i}F_a(\alpha^j)$ are all elements of $\Fq^*$,
because also $\beta^2$ is a primitive element of $\Fq$.
\end{proof}

We divide the study into two parts.

\begin{itemize}
  \item[$*$] Case $y=ax^2$. We intersect $\He$ with $y=ax^2$. As in (\ref{eqquad}),
we obtain $x^2(a^q x^{2q-2}-x^{q-1}+a)=0$. We set $x^{q-1}=t$
and we have to solve the equation $a^q t^2-t+a=0$. Setting $z=ta^q$ we obtain
$$z^2+z+a^{q+1}=0.$$

It is known that this equation has solutions in a field of characteristic even
if and only if $\Tr{\FQ}{\F}{a^{q+1}}=0$ \cite{CGC-alg-book-jungnickel1993}. And this latter condition holds, since we have
$$\Tr{\FQ}{\F}{a^{q+1}}=\Tr{\Fq}{\F}{\Tr{\FQ}{\Fq}{a^{q+1}}}=\Tr{\Fq}{\F}{0}=0.$$
We also have $\No(t)=1$, in fact $t^{q+1}=(x^{q-1})^{q+1}=1$, then we have
$$z^{q+1}=\No(z)=\No(a^q)=a^{q^2+q}=a^{q+1}=\No(a)$$
and so the equation becomes
$$z^2+z+z^{q+1}=0.$$
since $t\ne 0,z\ne 0$, then we must have
$$z^q+z=1.$$
We can note that, since $a^{q+1}\in\Fq$, then it is possible to calculate its trace
from $\Fq$ to $\F$, and we obtain
$$\Tr{\Fq}{\F}{a^{q+1}}=\Tr{\Fq}{\F}{z^{q+1}}=\Tr{\Fq}{\F}{z^2+z}=
z+z^2+z^2+z^4+\ldots+z^{q/2}+z^q=z+z^q;$$
If it is equal to $0$, we have a contradiction, then there is not any solution $x\in\FQ$.
On the other hand if it is equal to $1$, then we have solutions.

When the solutions exist, since $z^{q+1}=a^{q+1}$,
 a solution $z$ is $a\alpha^{j(q-1)}$, for some $j$,
and the other is $z+1$, which we can write as $a\alpha^{j'(q-1)}$.
From each of these we have the corresponding $t=(\frac{\alpha^j}{a})^{q-1}$ and so the 
$x$'s are $\frac{\alpha^{j+i(q+1)}}{a}=\frac{\alpha^j \beta^i}{a}$ and
$\frac{\alpha^{j'+i(q+1)}}{a}=\frac{\alpha^{j'}\beta^i}{a}$, with $i=0,\ldots,q-2$.

We can summarize:
\begin{itemize}
  \item[-] If $\Tr{\Fq}{\F}{a^{q+1}}=0$, there is only $one$ solution.
  This happens when $a=0$, which it is not acceptable, and for other $\frac{q}{2}-1$ values of
  $a^{q+1}$, so the possible values of $a$ are $(\frac{q}{2}-1)(q+1)$.
  \item[-] If $\Tr{\Fq}{\F}{a^{q+1}}=1$, there are $2q-1$ solutions. This happens
  for $\frac{q}{2}$ values of $a^{q+1}$, so the possible values of $a$ are
  $\frac{q}{2}(q+1)$.
\end{itemize}

As in the odd case, we apply the automorphism (\ref{autom}) to the parabolas
of type $y=ax^2$ and we can see that distinct automorphisms generate distinct parabolas.
We omit the easy adaption of our earlier proof.\\ 
So we have proved the following theorem:
\begin{theorem}\label{teo.intc0}
The Hermitian curve $\He$ and the parabola $y=ax^2$ intersects in either $one$ point or $2q-1$ points.\\ 
Moreover, from the application of (\ref{autom}) to these parabolas, we obtain:
\begin{description}
  \item[]$q^3(\frac{q}{2}-1)(q+1)$ parabolas with $one$ point of intersection with $\He$.
  \item[]$q^3\frac{q}{2}(q+1)$ parabolas with $2q-1$ points of intersection with $\He$.
\end{description}
\end{theorem}

\item[*] Case $y=ax^2+c$ with $\Tra(c)\ne 0$. We consider the equation (\ref{prin}).
We divide the problem into two parts:\\
\begin{itemize}
  \item[-] If $\Tr{\Fq}{\F}{a^{q+1}}=0$, we know that $F_a(x)$ is equal to zero only for $x=0$. 
  If $x\ne 0$, then by Lemma \ref{lem:abeven} if we fix $j$ we have that
  $F_a(x)=F_a(\alpha^j\beta^i)=\beta^{2i}F_a(\alpha^j)$ are all the elements of $\Fq^*$.
  But $j$ can assume $q+1$ distinct values, so any value of $\Fq^*$ can be obtained $q+1$ times.
  So, the equation $F_a(x)=\Tra(c)$ has exactly $q+1$ solutions.\\

  \item[-] If $\Tr{\Fq}{\F}{a^{q+1}}=1$, $F_a(x)=0$ has $2q-1$ solutions.
  So, if we fix an index $j$, the values of $F_a(\alpha^j\beta^i)=\beta^{2i}F_a(\alpha^j)$
  are all equal to zero or are all the elements of $\Fq^*$. There are exactly
  two distinct values of $j$ that give zero, so any non-zero value of $\Fq$ can be obtained $q-1$ times.
  So, the equation $F_a(x)=\Tra(c)$ has exactly $q-1$ solutions.
\end{itemize}

\noindent We apply the automorphism (\ref{autom}) to the parabola $y=ax^2+c$
and we obtain (\ref{parc}). These are all distinct and different from those of Theorem~\ref{teo.intc0},
because the planar intersection of $\He$ and the previous parabolas are different. 
The values of $c$ such that $\Tra(c)\ne 0$ are exactly $q^2-q$,
but we must be careful and not count twice the same parabola.
In particular, if two parabolas share a and b, then they are in the
same orbit if $\Tra(c)=\Tra(\bar c)$. So we must consider only one of these
for any non-zero value of $\Tra(c)$. These values are $q-1$.\\
Summarizing, we have proved the following theorem:
\begin{theorem}\label{teo.intcn0}
The Hermitian curve $\He$ and the parabola $y=ax^2+c$ with $\Tra(c)\ne 0$ intersect in  either $q+1$ or $q-1$ points.\\ 
Moreover, from the application of (\ref{autom}) to these parabolas, we obtain:
\begin{description}
  \item[]$q^3(\frac{q}{2}-1)(q+1)(q-1)$ parabolas (with $\Tr{\Fq}{\F}{a^{q+1}}=0$)
  with $q+1$ points of intersection with $\He$.
  \item[]$q^3\frac{q}{2}(q+1)(q-1)$ parabolas (with $\Tr{\Fq}{\F}{a^{q+1}}=1$)
  with $q-1$ points of intersection with $\He$.
\end{description}
\end{theorem}

\noindent  By summing all parabolas that we have found in Theorem \ref{teo.intc0} and Theorem \ref{teo.intcn0}, we obtain
$$q^3(q+1)\left(\frac{q}{2}-1+\frac{q}{2}+(q-1)(\frac{q}{2}-1+\frac{q}{2})\right)=$$
$$=q^3(q+1)(q-1)(1+q-1)=q^4(q^2-1).$$
Since this is exactly the total number of the parabolas, this means that we actually
considered all parabolas, and so we obtain the second half of Theorem \ref{teo.principe}.
\end{itemize}

%==================================================================

\section{Small-weight codewords of Hermitian Codes}
\label{codicihermitiani}
We recall that an affine-variety code is  $\I(\phi(L))$, where $\phi$ is as
(\ref{eval}). We consider a special case of affine-variety code, which is a Hermitian code.\\

Let $I=\langle y^q+y-x^{q+1},x^{q^2}-x,y^{q^2}-y\rangle\subset\FQ[x,y]$ and let
$R=\FQ[x,y]/I$. We take $L\subseteq R$ generated by
$$\Bmq=\{x^ry^s+I\mid qr+(q+1)s\leq m,\,\,0\leq s\leq q-1,\,\,0\leq r\leq q^2-1\},$$
where $m$ is an integer such that $0\leq m\leq q^3+q^2-q-2$. For simplicity, we also write
$x^ry^s$ for $x^ry^s+I$. We consider the evaluation map (\ref{eval})
$\phi:R\rightarrow(\FQ)^n$. We have the following affine--variety codes
$C(I,L)=\mathrm{Span}_{\FQ}\langle\phi(\Bmq)\rangle$
and we denote by $C(m,q)=(C(I,L))^{\perp}$ its dual.
Then the affine--variety code $C(m,q)$ is called the \textit{Hermitian code} with parity-check matrix $H$.
\begin{equation}\label{matpar}
  H=\left(
    \begin{array}{ccc}
      f_1(P_1) & \dots & f_1(P_n)\\
      \vdots & \ddots & \vdots\\
      f_i(P_1) & \dots & f_i(P_n)\\
    \end{array}
  \right)\textrm{ where }f_j\in\Bmq
\end{equation}
The Hermitian codes can be divided in four phases (\cite{CGC-cd-phdthesis-marco}), any of them
having specific explicit formulae linking their dimension and their distance, as in Table~\ref{Tab}.

{\scriptsize{
\begin{table}
\caption{The four \textquotedblleft phases\textquotedblright of Hermitian codes}\label{Tab}
\begin{center}
\begin{tabular}[h]{|cccc|}
\hline\noalign{\smallskip}
\textbf{Phase } & $\mathbf{m}$ & \textbf{Distance} $\mathbf{d}$ & \textbf{Dimension} $\mathbf{k}$\\
\noalign{\smallskip}
\hline
\noalign{\smallskip}
\textbf{1} & {\scriptsize$\begin{array}{c}
  0\leq m\leq q^2-2\\
  m=aq+b\\
  0\leq b\leq a\leq q-1\\ 
  b\leq q-2
\end{array}$} & {\scriptsize$\begin{array}{ll}
  a+1 & a>b\\
  a+2 & a=b
\end{array}\iff d\leq q$} & {\scriptsize$
  q^3-\frac{a(a+1)}{2}-(b+1)$}\\
  & & &\\
\textbf{2} & {\scriptsize$\begin{array}{c}
  q^2-1\leq m\leq 2q^2-2q-2\\
  m=2q^2-q-aq-b-1\\
  1\leq a\leq q-1\\
  0\leq b\leq q-1
\end{array}$} & {\scriptsize$\begin{array}{ll}
  (q-a)q-b &\,\,a<b\\
  (q-a)q &\,\,a\geq b
\end{array}$} & {\scriptsize$
  q^3-m+\frac{q(q-1)}{2}-1
$}\\
  & & &\\
\textbf{3} & {\scriptsize$2q^2-2q-1\leq m \leq q^3-1$} & {\footnotesize$m-q^2+q+1$}& {\scriptsize$q^3-m+\frac{q(q-1)}{2}-1$}\\
  & & &\\
\textbf{4} & {\scriptsize$\begin{array}{c}
  q^3\leq m < q^3+q^2-q-1\\
  m=q^3+aq+b\\
  0\leq a\leq q-1,\\
  0\leq b\leq q-2\\
\end{array}$} & {\scriptsize$\begin{array}{ll}
  q^3+aq+b-2g & a<b\\
  q^3+aq+a+1-2g & b\leq a
\end{array}$} & {\scriptsize$
  k\geq g-aq-b$}\\
\hline
\end{tabular}
\end{center}
\end{table}}}

\noindent In the remainder of this paper we focus on the first phase.
This case can be characterised by the condition $d\leq q$.

%%%%%%%%%%%%%%%%%%%%%%%%%%%%%%%%%%%%%%%%%%%%%%%%%%%%%%%%%%%%%%%%
%%%%%%%%%%%%%%%%%%%%%%%%%%%%%%%%%%%%%%%%%%%%%%%%%%%%%%%%%%%%%%%%
%%%%%%%%%%%%%%%%%%%%%%%%%%%%%%%%%%%%%%%%%%%%%%%%%%%%%%%%%%%%%%%%
%%%%%%%%%%%%%%%%%%%%%%%%%%%%%%%%%%%%%%%%%%%%%%%%%%%%%%%%%%%%%%%%
\break
\subsection{Corner codes and edge codes}\label{CCEC}
The first-phase Hermitian codes can be either \textit{edge codes} or \textit{corner codes}.
\begin{definition}\label{Cornercode}
Let $2\leq d\leq q$ and let $1\leq j\leq d-1$.\\
Let $L_0^d=\{1,x,\dots,x^{d-2}\},L_1^d=\{y,xy,\dots,x^{d-3}y\},\ldots,L_{d-2}^d=\{y^{d-2}\}$.\\
Let $l_1^d=x^{d-1},\ldots,l_j^d=x^{d-j}y^{j-1}$.
\begin{itemize}
  \item If $\Bmq=L_0^d\sqcup\dots\sqcup L_{d-2}^d$, then we say that $C(m,q)$ is a
  \textbf{corner code} and we denote it by $\textsf{H}^{\,0}_d$.
  \item If $\Bmq=L_0^d\sqcup\dots\sqcup L_{d-2}^d \sqcup\{l_1^d,\ldots,l_j^d\}$,
  then we say that $C(m,q)$ is an \textbf{edge code} and we denote it by $\textsf{H}^{\,j}_d$.
\end{itemize}
\end{definition}
From classical results in Table \ref{Tab} we have
\begin{theorem}
Let $2\leq d\leq q$, $1\leq j\leq d-1$. Then
$$d(\textsf{H}^{\,0}_d)=d(\textsf{H}^{\,j}_d)=d,\quad
\dim_{\FQ}(\textsf{H}^{\,0}_d)=n-\frac{d(d-1)}{2},\,\,
\dim_{\FQ}(\textsf{H}^{\,j}_d)=n-\frac{d(d-1)}{2}-j$$
\end{theorem}
\noindent In other words, all $\phi(x^ry^s)$ are linearly independent (i.e. $H$ has maximal rank)
and for any distance $d$ there are exactly $d$ Hermitian codes (one corner code and $d-1$ edge codes).
We can represent the above codes as in the following picture, where we consider the five smallest
non-trivial codes (for any $q\geq 3$).

\begin{minipage}[c]{.55\textwidth}
\centering
\begin{itemize}
  \item [$\mathbf{\textsf{H}^{\,0}_{2}}$] is a [$n,n-1,2$] code.\\
  $\Bmq=L_0^2=\{1\}$, so the parity-check matrix of $\textsf{H}^{\,0}_{2}$ is $(1,\dots,1)$.
  \item[$\mathbf{\textsf{H}^{\,1}_{2}}$] is a [$n,n-2,2$] code.\\
  $\Bmq=L_0^2\sqcup l_1^2=\{1,x\}$
  \item[$\mathbf{\textsf{H}^{\,0}_{3}}$] is a [$n,n-3,3$] code.\\
  $\Bmq=L_0^3\sqcup L_1^3 =\{1,x,y\}$
  \item[$\mathbf{\textsf{H}^{\,1}_{3}}$] is a [$n,n-4,3$] code.\\
  $\Bmq=L_0^3\sqcup L_1^3\sqcup l_1^3=\{1,x,y,x^2\}$
  \item[$\mathbf{\textsf{H}^{\,2}_{3}}$] is a [$n,n-5,3$] code.\\
  {\small{$\Bmq=L_0^3\sqcup L_1^3\sqcup\{l_1^3,l_2^3\}=\{1,x,y,x^2,xy\}$}}
\end{itemize}
\end{minipage}
\begin{minipage}[c]{.45\textwidth}
\centering
\includegraphics[width=5.5 cm]{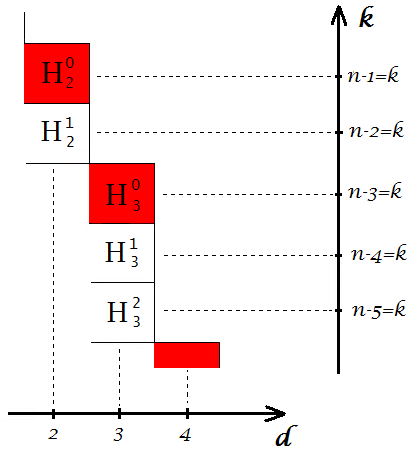}
\end{minipage}

\subsection{First results for the first phase}\label{parolemin}

Ideal $J_w$ of Proposition \ref{proptot} for $C(m,q)$ is
\begin{equation}\label{Jqm}
\begin{array}{ll}
  J_{w}=\Big\langle & \Big\{\sum_{i=1}^wz_ix_i^ry_i^s\Big\}_{x^ry^s\in\mathcal{B}_{m,q}},
  \left\{x_i^{q+1}-y_i^q-y_i\right\}_{i=1,\dots,w},\\
  & \left\{z_i^{q^2-1}-1\right\}_{i=1,\dots,w},\left\{x_i^{q^2}-x_i\right\}_{i=1,\dots,w},
  \left\{y_i^{q^2}-y_i\right\}_{i=1,\dots,w},\\
  & \left\{\prod_{1\leq i<j\leq w} ((x_i-x_j)^{q^2-1}-1) ((y_i-y_j)^{q^2-1}-1\right\}\Big\rangle.
\end{array}
\end{equation}

Let $w\geq v\geq 1$. Let
$Q=(\overline{x}_1,\dots,\overline{x}_w,\overline{y}_1,\dots,\overline{y}_w,\overline{z}_1,\dots,\overline{z}_w)\in\mathcal{V}(J_w)$.
We say that $Q$ is in \textbf{v-block position} if we can partition $\{1,\ldots,n\}$ in $v$ blocks
$I_1,\ldots,I_v$ such that
$$\overline{x}_i=\overline{x}_j\iff\exists\,1\leq h\leq v\textrm{ such that }i,j\in I_h.$$
W.l.o.g. we can assume $|I_1|\leq\dots\leq |I_v|$ and $I_1=\{1,\dots,u\}$.
It is simple to prove the following numerical lemma.

\begin{lemma}\label{lemma1}
We always have $u+v\leq w+1$. If $u\geq 2$ and $v\geq 2$, then
$v\leq\lfloor\frac{w}{2}\rfloor$ and $u+v\leq\lfloor\frac{w}{2}\rfloor+2$.
\end{lemma}

We need the following technical lemma \cite{CGC-cd-phdthesis-marco}.

\begin{lemma}\label{prop2}
Let us consider the edge code $\textsf{H}^{\,j}_d$ with $1\leq j\leq d-1$ and $3\leq d\leq w\leq 2d-3$.
Let $Q=(\overline{x}_1,\dots,\overline{x}_w,\overline{y}_1,\dots,\overline{y}_w,\overline{z}_1,
\dots,\overline{z}_w)$
be a solution of $J_{w}$ in $v$-block position, with $v\leq w$, then exactly one of the following cases holds:
\begin{itemize}
  \item[(a)] $u=1\qquad v>d\qquad w\geq d+1$
  \item[(b)] $v=1$, that is, $\bar{x}_1=\dots=\bar{x}_w$
\end{itemize}
If $d=2$ and $w=2$, then (a) holds for $H^1_2$.
\end{lemma}
\begin{proof}
We denote for all $1\leq h\leq v$
$$
X_h=\bar{x}_i\textrm{ if }i\,\in\,I_h,\quad
Z_h=\sum_{i\in I_h}\bar{z}_i,\quad
Y_{h,\delta}=\sum_{i\in I_h}\bar{y}_i^{\delta}\bar{z}_i\textrm{ with }1\leq\delta\leq u-1
$$
\begin{itemize}
  \item[(a)] $u=1$. We have to prove, by contradiction, that $v>d$.\\
Let $v\leq d$. Since $Q\in\mathcal{V}(J_w)$, then $L_0^w(Q)=l_1^w(Q)=0$, that is
\begin{equation}\label{puntoa}
0=\sum_{i=1}^w\bar{x}_i^r\bar{z}_i=\sum_{i\in I_h}X_h^r\bar{z}_i=\sum_{h=1}^v X_h^r Z_h\quad 0\leq r\leq d-1.
\end{equation}
We can consider only the first $v$ equations of (\ref{puntoa}), because $v\leq d$, so
\begin{equation}\label{sistX}
  \sum_{h=1}^v X_h^r Z_h=0\quad 0\leq r\leq v-1\iff\left(\begin{array}{ccc}
  1 & \dots & 1 \\
  X_1 & \dots & X_v\\
  \vdots & \dots & \vdots\\
  X_1^{v-1} & \dots & X_v^{v-1}
\end{array}\right)\left(\begin{array}{c}
  Z_1\\
  \vdots\\
  Z_v
\end{array}\right)=0
\end{equation}
The above matrix is a Vandermonde matrix, so it has maximal rank $v$. Therefore, the solution of
(\ref{sistX}) is $(Z_1,\dots,Z_v)=(0,\dots,0)$. Since $u=1$, then $Z_1=\overline{z}_1=0$,
which contradicts $\overline{z}_i\in\FQ\setminus\{0\}$. So if $v>d$ then $w\geq d+1$.\\
\item[(b)] $u\geq 2$. We suppose by contradiction that $v\geq 2$.\\
We consider Proposition \ref{proptot}. A subset of equations of condition (\ref{e1})
is the following system, where $0\leq r\leq v$
\begin{equation}\label{puntob}
  \left\{\begin{array}{l}
    \sum_{i=1}^w\bar{x}_i^r\bar{z}_i=0\\
    \sum_{i=1}^w\bar{x}_i^r\bar{y}_i\bar{z}_i=0\\
    \quad\vdots\\
    \sum_{i=1}^w \bar{x}_i^r{\bar{y}_i}^{u-1}\bar{z}_i=0 
  \end{array}
\right.\iff\left\{
  \begin{array}{l}
    \sum_{h=1}^v X_h^r Z_h=0\\
    \sum_{h=1}^v X_h^r Y_{h,1}=0\\
    \quad\vdots\\
    \sum_{h=1}^v X_h^r Y_{h,u-1}=0
  \end{array}
\right.
\end{equation}
In fact system (\ref{puntob}) is a subset of (\ref{e1}) if and only if $\deg(\bar{x}_i^{v}\bar{y}_i^{u-1})\leq d-1$
for any $i=1,\dots,w$. That is, $v+(u-1)\leq d-1\iff v+u\leq d$.\\ 
To verify it, since $v\geq 2$,
it is sufficient to apply Lemma \ref{lemma1} and we obtain
$u+v\leq\lfloor\frac{w}{2}\rfloor+2\leq\lfloor\frac{2d-3}{2}\rfloor+2=d$.\\
By system (\ref{puntob}) we obtain $u$ Vandermonde matrices (all having rank $v$).
Therefore, the solutions of these systems are zero-solutions.
So, in the particular case $h=1$, we have $Z_1=Y_{1,1}=\ldots=Y_{1,u-1}=0$, that is
$$\left\{\begin{array}{l}
  \sum_{i=0}^u\bar{z}_i=0\\
  \sum_{i=0}^u\bar{y}_i\bar{z}_i=0\\
  \quad\vdots\\
  \sum_{i=0}^u\bar{y}_i^{u-1}\bar{z}_i=0
\end{array}\right.\iff\left(\begin{array}{ccc}
  1 & \dots & 1\\
  \bar{y}_1 & \dots & \bar{y}_u\\
  \vdots & \dots & \vdots\\
  \bar{y}_1^{u-1} & \dots & \bar{y}_u^{u-1}
\end{array}\right)\left(\begin{array}{c}
  \bar{z}_1\\
  \vdots\\
  \bar{z}_u
\end{array}\right)=0$$
Since the $\bar{y}_i$'s are all distinct (because the $\bar{x}_i$'s are all equal),
we obtain a Vandermonde matrix, and so $\bar{z}_1=\dots=\bar{z}_u=0$, but it is impossible because
$\bar{z}_i\in \FQ\setminus\{0\}$. Therefore $v=1$.
\end{itemize}
The case $H^1_2$ is trivial.
\end{proof}

%%%%%%%%%%%%%%%%%%%%%%%%%%%%%%%%%%%%%%%%%%%%%%%%%%%%%%%%%%%%%%%%%%%%%%%%%%%%%
%%%%%%%%%%%%%%%%%%%%%%%%%%%%%%%%%%%%%%%%%%%%%%%%%%%%%%%%%%%%%%%%%%%%%%%%%%%%%
%%%%%%%%%%%%%%%%%%%%%%%%%%%%%%%%%%%%%%%%%%%%%%%%%%%%%%%%%%%%%%%%%%%%%%%%%%%%%
%%%%%%%%%%%%%%%%%%%%%%%%%%%%%%%%%%%%%%%%%%%%%%%%%%%%%%%%%%%%%%%%%%%%%%%%%%%%%

\subsection{Minimum-weight codewords}\label{minimum}

This is the first result relating intersection of lines with $\He$ and small-weight codewords presented in  \cite{CGC-cod-misc-salpell06}.
\begin{corollary}\label{cor1}
Let us consider the edge code $\textsf{H}^{\,j}_d$ with $1\leq j\leq d-1$.\\
If $Q=(\overline{x}_1,\dots,\overline{x}_d,\overline{y}_1,\dots,\overline{y}_d,\overline{z}_1,\dots,\overline{z}_d)\in\mathcal{V}(J_{d})$,
then $\bar{x}_1=\dots=\bar{x}_d$. In other words, the points that correspond to a minimum-weight word lie
in the intersection of the Hermitian curve $\He$ and a vertical line.\\
Whereas if $d\geq 4$ and
$Q=(\overline{x}_1,\dots,\overline{x}_{d+1},\overline{y}_1,\dots,\overline{y}_{d+1},\overline{z}_1,\dots,\overline{z}_{d+1})\in\mathcal{V}(J_{d+1})$,
then one of the following cases holds
\begin{itemize}
  \item[(a)] $\bar{x}_i\ne \bar{x}_{j}$ with $i\ne j$ for $1\leq i,j\leq d+1$
  \item[(b)] $\bar{x}_1=\dots=\bar{x}_{d+1}$.
\end{itemize}
\end{corollary}
\begin{proof}
We are in the hypotheses of Lemma \ref{prop2}. So if $w=d$ then $u\ne 1$. So $v=1$.
Whereas, if $w=d+1$ then there are two possibilities. In case $(a)$ of Lemma~\ref{prop2},
all the $\bar{x}_i$'s are different, since $v=d+1$, or, case $(b)$, $\bar{x}_1=\dots=\bar{x}_{d+1}$.
\end{proof}

Now we can prove the following theorem for edge codes.
\begin{theorem}\label{teo1}
The number of minimum weight words of an edge code $\textsf{H}^{\,j}_d$ is
$$A_d=q^2(q^2-1)\binom{q}{d}.$$
\end{theorem}
\begin{proof}
By Proposition \ref{proptot} we know that $J_{d}$  represents all words of minimum weight.
The first set of ideal basis (\ref{Jqm}) has exactly $\frac{d(d-1)}{2}+j$ equations, where $1\leq j\leq d-1$.
So, if $j=1$, this set implies the following system:
{\small{\begin{equation}\label{sistC}
\left\{\begin{array}{l}
  \bar{z}_1+\dots+\bar{z}_d=0\\
  \bar{x}_1\bar{z}_1+\dots+\bar{x}_d\bar{z}_d=0\\
  \bar{y}_1\bar{z}_1+\dots+\bar{y}_d\bar{z}_d=0\\
  \bar{x}^2_1\bar{z}_1+\dots+\bar{x}^2_d\bar{z}_d=0\\
  \quad\vdots\\
  \bar{y}_1^{d-2}\bar{z}_1+\dots+\bar{y}_d^{d-2}\bar{z}_d=0\\
  \bar{x}_1^{d-1}\bar{z}_1+\dots+\bar{x}_d^{d-1}\bar{z}_d=0\\
\end{array}\right.
\end{equation}}}
Whereas, if $j>1$ then we have to add the first $j-1$ of following equations:
{\small{\begin{equation}\label{sist:j}
\left\{\begin{array}{l}
  \bar{x}_1^{d-2}\bar{y}_1\bar{z}_1+\dots+\bar{x}_d^{d-2}\bar{y}_d\bar{z}_d=0\\
  \quad\vdots\\
  \bar{x}_1 \bar{y}_1^{d-2}\bar{z}_1+\dots+\bar{x}_d\bar{y}_d^{d-2}\bar{z}_d=0\\
\end{array}\right.
\end{equation}}}
\noindent But $\bar{x}_1=\ldots=\bar{x}_d$, since we are in the hypotheses of Corollary~\ref{cor1}. So the system becomes
{\small{\begin{equation}\label{sistD}
\left\{\begin{array}{l}
  \bar{z}_1+\dots+\bar{z}_d=0\\
  \bar{y}_1\bar{z}_1+\dots+\bar{y}_d\bar{z}_d=0\\
  \quad\vdots\\
  \bar{y}_1^{d-2}\bar{z}_1+\dots+\bar{y}_d^{d-2}\bar{z}_d=0\\
\end{array}\right.
\end{equation}}}
We have $q^2$ choice for the $\bar{x}_i$'s and, by Lemma \ref{interette}, we have $\binom{q}{d}d!$
different $\bar{y}_i$'s, since for any choice of the $\bar{x}_i$'s there are exactly $q$ possible values
for the $\bar{y}_i$'s, but we need just $d$ of them and any permutation of these will be again a solution.
Now we have to calculate the solutions for the $\bar{z}_i$'s.\\
\noindent We write the system (\ref{sistD}) as a matrix, which is a Vandermonde matrix with rank $d-1$.
This means that the solution space has linear dimension $1$ because $1=d-(d-1)=$ number of variables $-$
rank of matrix. So the solutions are $(a_1\alpha,a_2\alpha,\dots,a_{d-1}\alpha)$ with $\alpha\in\,\FQ^*$,
where $a_j$ are fixed since they depend on $\bar{y}_i$. So the number of the $z$'s is
$|\FQ^*|=q^2-1$, then $A_d=\frac{1}{d!}\left(q^2(q^2-1)\binom{q}{d}d\,!\right).$
\end{proof}

We consider now corner codes. We have the following geometric characterisation.
\begin{proposition}\label{cornercode}
Let us consider the corner code $\textsf{H}^{\,0}_d$, then the points\\
{\small{$(\bar{x}_1,\bar{y}_1),\ldots,(\bar{x}_d,\bar{y}_d)$}}
corresponding to minimum-weight words lie on the same line.
\end{proposition}
\begin{proof}
The minimum-weight words of a  corner code have to verify the first condition set of $J_w$, which has $\frac{d(d-1)}{2}$ equations. That is,
{\small{\begin{equation}\label{sistE}
\left\{\begin{array}{l}
  \bar{z}_1+\dots+\bar{z}_d=0\\
  \bar{x}_1\bar{z}_1+\dots+\bar{x}_d\bar{z}_d=0\\
  \bar{y}_1\bar{z}_1+\dots+\bar{y}_d\bar{z}_d=0\\
  \bar{x}^2_1\bar{z}_1+\dots+\bar{x}^2_d\bar{z}_d=0\\
  \quad\vdots\\
  \bar{y}_1^{d-2}\bar{z}_1+\dots+\bar{y}_d^{d-2}\bar{z}_d=0\\
\end{array}\right.
\end{equation}}}
\noindent This system is the same as (\ref{sistC}), but with a missing equation. This means that
(\ref{sistE}) have all solutions of system (\ref{sistC}) and other solutions.\\
If we consider a subset of (\ref{sistE}):
{\small{\begin{equation}\label{sistF}
\left\{\begin{array}{l}
  \bar{z}_1+\dots+\bar{z}_d=0\\
  \bar{x}_1\bar{z}_1+\dots+\bar{x}_d\bar{z}_d=0\\
  \bar{x}^2_1\bar{z}_1+\dots+\bar{x}^2_d\bar{z}_d=0\\
  \quad\vdots\\
  \bar{x}_1^{d-2}\bar{z}_1+\dots+\bar{x}_d^{d-2}\bar{z}_d=0\\
\end{array}\right.
\end{equation}}}
\noindent we note that the $\bar{z}_i$'s are all non-zero if all $\bar{x}_i$'s are distinct
(or all are equal). Therefore, we have only two possibilities for the $\bar{x}_i$'s: either
are all different or they coincide. The same consideration is true for the $\bar{y}_i$'s,
in fact when we consider (\ref{sistE}) and we exchange $x$ with $y$, we obtain again (\ref{sistE}).\\
So we have an alternative:
\begin{itemize}
  \item The $\bar{x}_i$'s are all equal or the $\bar{y}_i$'s are all equal, so our proposition is true.
  \item The $\bar{x}_i$'s and the $\bar{y}_i$'s are all distinct. We prove that  lye in the
intersection of a non-horizontal line.\\
Let $y=\beta x+\lambda$ be a non-vertical line passing for two points in a minimum weight
configuration. We can do an affine transformation of this type:
$$\left\{\begin{array}{l}
  x=x'\\
  y=y'+ax'\quad a\in\FQ
\end{array}\right.$$
such that at least two of the $y'$'s are equal and not all $y$'s are coincident.
Substituting the above transformation in (\ref{sistE}) and applying some operations between the equations,
we obtain a system that is equivalent to (\ref{sistE}). But this new system have all $y'$'s equal
(or all distinct), so the $y'$'s have to be all equal. Hence we can conclude that the points lie on the same line.
\end{itemize}
\end{proof}

We finally prove the following theorem:
\begin{theorem}\label{teo2}
The number of words having weight $d$ of a corner code $\textsf{H}^{\,0}_d$ is
$$A_d=q^2(q^2-1)\binom{q}{d-1}\frac{q^3-d+1}{d}.$$
\end{theorem} 
\begin{proof}
Again, the points corresponding to minimum-weight words of a corner code have to verify (\ref{sistE}). By above
proposition, we know that these points lie in the intersections of any line and the Hermitian curve $\He$.\\

Let $Q=(\overline{x}_1,\dots,\overline{x}_d,\overline{y}_1,
\dots,\overline{y}_d,\overline{z}_1,\dots,\overline{z}_d)\in\mathcal{V}(J_{d})$
such that $\overline{x}_1=\ldots=\overline{x}_d$, that is, the points $(\bar{x}_i,\bar{y}_i)$ lie on a vertical line.
We know that the number of such $Q$'s is 
$$
q^2(q^2-1)\binom{q}{d}d\,!\;.
$$ 
Now we have to calculate the number
of solutions $Q\in\mathcal{V}(J_{d})$ such that $(\bar{x}_i,\bar{y}_i)$ lie on a non-vertical line.\\
\noindent By Lemma \ref{intertot} we know that the number of the $\bar{y}_i$'s and $\bar{x}_i$'s is
$(q^4-q^3)\binom{q+1}{d}d\,!$, since for any choice of the $\bar{y}_i$'s there are exactly $q+1$
possible values for the $\bar{x}_i$'s, but we need just $d$ of this (and the system is invariant).
As regards the number of the $\bar{z}_i$'s, we have to calculate the number of solutions of system (\ref{sistE}).\\
We apply an affine transformation to the system (\ref{sistE}) to obtain a horizontal line, that is, to have all the $\bar{x}_i$'s different and all the $\bar{y}_i$'s are equal, so we obtain a system equivalent to system (\ref{sistD}). Therefore we have a Vandermonde matrix, hence the number of the
$\bar{z}_i$'s is $q^2-1$. So
{\small{$$\begin{array}{rl}
  A_d= & \frac{1}{d!}\left(q^2(q^2-1)\binom{q}{d}d\,!+(q^4-q^3)(q^2-1)\binom{q+1}{d}d\,!\right)\\
  = & q^2(q^2-1)\binom{q}{d-1}\frac{q^3-d+1}{d}.
\end{array}$$}}
\end{proof}

\subsection{Second-weight codewords}\label{secondweight}
In this section we state more theorems for edge and corner codes previously consider in \cite{CGC-alg-tesi3-marcolla08}.
We study the case when the $x_i$'s coincide or when the $y_i$'s coincide.

\begin{theorem}\label{teo:d+1}
The number of words of weight $d+1$ with $y_1=\ldots=y_{d+1}$ of a corner code $\textsf{H}^{\,0}_d$ is:
$$(q^2-q)(q^4-(d+1)q^2+d)\binom{q+1}{d+1}.$$
Whereas for an edge code $\textsf{H}^{\,j}_d$ with $1\leq j\leq d-1$ is:
$$(q^2-q)\binom{q+1}{d+1}.$$
\end{theorem}

\begin{proof}
We have $q^2$ choice for the $\bar{y}_i$'s and, by Corollary \ref{interettey}, we have\\ $\binom{q+1}{d+1}(d+1)!$
different $\bar{x}_i$'s, since for any choice of the $\bar{y}_i$'s there are exactly $q+1$ possible values
for the $\bar{x}_i$'s, but we need just $(d+1)$ of them and any permutation of these will be again a solution.\\
Now we have to calculate the solutions for the $\bar{z}_i$'s, in the two distinct cases.\\
\begin{itemize}
  \item[$*$]\textbf{Case} $\textsf{H}^{\,0}_d$. By Proposition \ref{proptot} we know that $J_{d}$ represents all words of minimum weight. The first set of ideal basis (\ref{Jqm}) has exactly $\frac{d(d-1)}{2}$ equations, which is system (\ref{sistE}) with more variables, that is, instead of
$\bar{x}_{d}$, $\bar{y}_{d}$ and $\bar{z}_{d}$, we have, respectively, $\bar{x}_{d+1}$, $\bar{y}_{d+1}$
and $\bar{z}_{d+1}$. Since $\bar y_1=\ldots=\bar y_{d+1}$, the said variation of system (\ref{sistE}) is
\begin{equation}\label{cornerrid}
\left\{\begin{array}{l}
  \bar{z}_1+\dots+\bar{z}_{d+1}=0\\
  \bar{x}_1\bar{z}_1+\dots+\bar{x}_{d+1}\bar{z}_{d+1}=0\\
  \bar{x}^2_1\bar{z}_1+\dots+\bar{x}^2_{d+1}\bar{z}_{d+1}=0\\
  \quad\vdots\\
  \bar{x}_1^{d-2}\bar{z}_1+\dots+\bar{x}_{d+1}^{d-2}\bar{z}_{d+1}=0\\
\end{array}\right.
\end{equation}
We can note that, if we write the system (\ref{cornerrid}) as a matrix adding these two equations
$x_1^{d-1}+\ldots+x_{d+1}^{d-1}=0$ and $x_1^{d}+\ldots+x_{d+1}^{d}=0$ we obtain a Vandermonde matrix.
So all rows of (\ref{cornerrid}) are linearly independent. This means that the solution space
has linear dimension $2$ because $2=(d+1)-(d-1)=$ number of variables $-$ rank of matrix.
So the number of the $z$'s is
$$
q^2(\mbox{ for }z_{d+1})\cdot q^2(\mbox{ for }z_{d})-\#(z_i=0\textrm{ for at least an }i).
$$
We want to calculate the number of $z_i=0$ for at least one $i$.\\
Since the matrix  $\mathbf{H}$ has maximum rank, we can apply the Gauss elimination to the system (\ref{cornerrid})
\begin{equation}\label{sisth}
\left\{\begin{array}{l}
  \bar{z}_1+ \dots + \bar{z}_{d+1}=0\\
  h_{2,2}\bar z_2+\ldots+ h_{2,d}\bar z_d + h_{2,d+1}+\bar z_{d+1}=0\\
  \vdots\\
  h_{d-1,d-2}\bar z_{d-2}+h_{d-1,d-1}\bar z_{d-1}+ h_{d-1,d}\bar z_d+  h_{d-1,d+1}\bar z_{d+1}=0 \\
  h_{d-1,d-1}\bar z_{d-1}+ h_{d-1,d}\bar z_d+  h_{d-1,d+1}\bar z_{d+1}=0 
\end{array}\right.
\end{equation}
If we solve the system (\ref{sisth}) we obtain
\begin{equation}\label{z234h}
  h_{d-1,d-1}\bar z_{d-1}+h_{d-1,d}\bar z_d+h_{d-1,d+1}\bar z_{d+1}=0
\end{equation}
First of all we consider the case $\bar z_{d-1}=0$, that is
\begin{equation}\label{z20h}
  h_{d-1,d}\bar z_d+h_{d-1,d+1}\bar z_{d+1}=0\iff\bar z_d=-\frac{h_{d-1,d+1}}{h_{d-1,d}}\bar z_{d+1}
\end{equation}
The equation (\ref{z20h}) in the variable $\bar z_{d+1}\in\FQ$ has exactly $q^2$ solutions.
In particular, we have the pair $(\bar z_d,\bar z_{d+1})=(0,0)$ and other $q^2-1$ ways to choose
the variable $\bar z_{d+1}$.\\ 
We have similar condition when $\bar z_d=0$ and $\bar z_{d+1}=0$. As before we have the pairs
$(\bar z_{d-1},\bar z_{d+1})=(0,0)$ and $(\bar z_{d-1},\bar z_d)=(0,0)$ and other $q^2-1$ ways
to choose $\bar z_{d-1}$ and $q^2-1$ ways to choose $\bar z_d$.\\
So the equation (\ref{z234h}) has exactly
$3(q^2-1)+|\{(\bar z_{d-1},\bar z_d,\bar z_{d+1})=(0,0,0)\}|=3q^2-2$ solutions.\\
Now we consider the second last line of the system (\ref{sisth}):
$h_{d-1,d-2}\bar z_{d-2}+h_{d-1,d-1}\bar z_{d-1}+h_{d-1,d}\bar z_d+h_{d-1,d+1}\bar z_{d+1}=0$, that is,
\begin{equation}\label{z1234h}
  \bar z_{d-2}=-(k_{d-1}\bar z_{d-1}+k_{d}\bar z_d+k_{d+1}\bar z_{d+1})
\end{equation}
First of all we have to study the case $\bar z_{d-2}=0$. We just studied the case in which all variables
$\bar z_{d-1}=\bar z_d=\bar z_{d+1}=0$, so we have to study the case when all variables
are different from zero, that is,
$\frac{k_{d-1}}{k_{d+1}}\bar z_{d-1}+\frac{k_{d}}{k_{d+1}}\bar z_d=-\bar z_{d+1}=k\in\FQ^*$.
So the equation (\ref{z1234h}) has exactly $(q^2-1)$ solutions.\\
We repeat the argument for each of system's equations (\ref{sisth}), that are $(d-2)$,
if we do not count the last equation. Therefore
$$\#(\bar z_i=0\textrm{ for at least one }i)=3q^2-2+(d-2)(q^2-1)=(d+1)q-d$$
So the system (\ref{sisth}) has exactly $q^4-(d+1)q+d$ solutions, then the number of words of weight $d+1$ with $y_1=\ldots=y_{d+1}$ of $\textsf{H}^{\,0}_d$ is:
$$
(q^2-q)(q^4-(d+1)q^2+d)\binom{q+1}{d+1}.
$$
  \item[$*$]\textbf{Case} $\textsf{H}^{\,j}_d$.
In this case the first set of ideal basis (\ref{Jqm}) contains exactly $\frac{d(d-1)}{2}+j$ equations, where $1\leq j\leq d-1$.
So, if $j=1$, this set implies the system (\ref{sistC}) with more variables, that is, instead of $\bar{x}_{d}$, $\bar{y}_{d}$ and $\bar{z}_{d}$, we have, respectively, $\bar{x}_{d+1}$, $\bar{y}_{d+1}$ and $\bar{z}_{d+1}$.
Whereas, if $j>1$ then we have to add the first $j-1$ of equations (\ref{sist:j}) with more variables.\\
Since $\bar y_1=\ldots=\bar y_{d+1}$, the system becomes
{\small{\begin{equation}
\left\{\begin{array}{l}
  \bar{z}_1+\dots+\bar{z}_{d+1}=0\\
  \bar{x}_1\bar{z}_1+\dots+\bar{x}_{d+1}\bar{z}_{d+1}=0\\
  \bar{x}^2_1\bar{z}_1+\dots+\bar{x}^2_{d+1}\bar{z}_{d+1}=0\\
  \quad\vdots\\
  \bar{x}_1^{d-1}\bar{z}_1+\dots+\bar{x}_{d+1}^{d-1}\bar{z}_{d+1}=0\\
\end{array}\right.
\end{equation}}}
\noindent This means that the solution space has linear dimension $d-(d-1)=1$.
So the number of the $z$'s is $|\FQ^*|=q^2-1$, then the number of words of weight $d+1$ with $y_1=\ldots=y_{d+1}$ of $\textsf{H}^{\,j}_d$ is:
$$(q^2-1)(q^2-q)\binom{q+1}{d+1}.$$
\end{itemize}
\end{proof}

\begin{theorem}
The number of words of weight $d+1$ with $x_1=\ldots=x_{d+1}$ of a corner code $\textsf{H}^{\,0}_d$
and of an edge code $\textsf{H}^{\,j}_d$ is:
$$q^2(q^4-(d+1)q^2+d)\binom{q}{d+1}.$$
\end{theorem}
\begin{proof}
By Proposition \ref{proptot} we know that $J_{d}$  represents all words of minimum weight.
For an edge code the first set of ideal basis (\ref{Jqm}) implies, if $j=1$ the system (\ref{sistC})
with more variables\footnote{instead of $\bar{x}_{d}$, $\bar{y}_{d}$ and $\bar{z}_{d}$, we have,
respectively, $\bar{x}_{d+1}$, $\bar{y}_{d+1}$ and $\bar{z}_{d+1}$. This is true every time that we write
with more variables} and if $j>1$ we have to add the first $j-1$ of equations (\ref{sist:j})
with more variables. Whereas, for a corner code, the first set of ideal basis (\ref{Jqm})
implies the system (\ref{sistE}) with more variables.
\noindent But $\bar{x}_1=\ldots =\bar{x}_{d+1}$, so the systems becomes
{\small{\begin{equation}
\left\{\begin{array}{l}
  \bar{z}_1+\dots+\bar{z}_{d+1}=0\\
  \bar{y}_1\bar{z}_1+\dots+\bar{y}_{d+1}\bar{z}_{d+1}=0\\
  \quad\vdots\\
  \bar{y}_1^{d-2}\bar{z}_1+\dots+\bar{y}_{d+1}^{d-2}\bar{z}_{d+1}=0\\
\end{array}\right.
\end{equation}}}
We have $q^2$ choice for the $\bar{x}_i$'s and, by Lemma \ref{interette}, we have $\binom{q}{d+1}(d+1)!$ different
$\bar{y}_i$'s, since for any choice of the $\bar{x}_i$'s there are exactly $q$ possible values for the
$\bar{y}_i$'s, but we need just $d+1$ of them and any permutation of these will be again a solution.
And we have $(q^4-(d+1)q^2+d)$ possible $\bar{z}_i$'s which is exactly the situation met in Theorem~\ref{teo:d+1}.\\
\end{proof}
 
\begin{theorem}\label{teo2:d+1}
The number of words of weight $d+1$ of a corner code $\textsf{H}^{\,0}_d$ with $(x_i,y_i)$ lying on a non-vertical line is:
$$(q^4-q^3)(q^4-(d+1)q^2+d)\binom{q+1}{d+1}.$$
\end{theorem}

\begin{theorem}\label{teo3:d+1}
The number of words of weight $d+1$ of an edge code $\textsf{H}^{\,j}_d$ with $(x_i,y_i)$ lying on a non-vertical line is:
$$(q^4-q^3)(q^2-1)\binom{q+1}{d+1}.$$
\end{theorem}

The proofs are similar to those of the statements as in Section \ref{parolemin} and the previous theorems and so are omitted.\\
In other cases, we have to consider the intersection of the curve with higher degree curves and the formulae get more complicated. For example the cubic found in \cite{CGC-cod-art-couvreur2011dual,CGC-cod-art-ballico2012goppa}.\\

Now we are going to study some special cases of Hermitian codes, that is, we count the number of words having weight $d+1$ for any Hermitian code having distance $d=3$ or $d=4$. In the following subsection we are going to prove these theorems:

\begin{theorem}\label{teo.H3}
The number of words of weight $4$ of a corner code $\textsf{H}^{\,0}_3$ is:
$$A_4=\frac{1}{4}\left(\binom{q^3}{3}(q+1)-q^2\binom{q+1}{3}
(3q^3+2q^2-8)\right)(q-1)(q^3-3).$$
The number of words of weight $4$ of an edge code $\textsf{H}^{\,1}_3$ is:
$$A_{4}=q^2\binom{q}{4}(q^4-4q^2+3)+\frac{q^4(q^2-1)^2(q-1)^2}{8}+
(q^2-1)\sum_{k=4}^{2q}N_k\binom{k}{4}.$$
Where $N_k$ is the number of parabolas and non-vertical lines that intersect $\He$
in exactly $k$ points.\\
The number of words of weight $4$ of an edge code $\textsf{H}^{\,2}_3$ is:
$$A_{4}=q^2(q-1)\binom{q+1}{4}(2q^3-3q^2-4q+9).$$
\end{theorem}

\begin{theorem}\label{teo.H4}
The number of words of weight $5$ of a corner code $\textsf{H}^{\,0}_4$ is:
$$A_5=\frac{1}{5}q^2\binom{q}{4}(q^3-4)(q^2-1)(q^2-4).$$
The number of words of weight $5$ of all edge codes $\textsf{H}^{\,j}_4$ for $1\leq j\leq 3$ is:
$$A_5=q^2(q-1)\binom{q+1}{5}(2q^3-4q^2-5q+16).$$
\end{theorem}

The formula for $A_{4}$ of $\textsf{H}^{\,1}_3$ in Theorem \ref{teo.H3} contains some implicit values $N_k$'s.
To derive explicit values it is enough to consider Theorem \ref{teo.principe}.\\

\break

\subsection{The complete investigation for $d=3,4$.}\label{complete}
In this subsection we will study separately the following cases:
$\textnormal{H}^0_3$, $\textnormal{H}^1_3$, $\textnormal{H}^2_3$,
$\textnormal{H}^0_4$, $\{\textnormal{H}^{\, j}_4\}_{1\leq j\leq 3}$.\\

%%%%%%%%%%%%%%%%%%%%%%%%%%%%%%%%%%%%%%%%%%%%%%%%%%%%%%%%%%%%%%%%%%%%%%%
%%%%%%%%%%%%%%%%%%%%%%%%%%%%%%%%%%%%%%%%%%%%%%%%%%%%%%%%%%%%%%%%%%%%%%%
%%%%%%%%%%%%%%%%%%%%%%%%%%%%%%%%%%%%%%%%%%%%%%%%%%%%%%%%%%%%%%%%%%%%%%%
%%%%%%%%%%%%%%%%%%%%%%%%%%%%%%%%%%%%%%%%%%%%%%%%%%%%%%%%%%%%%%%%%%%%%%%
%%%%%%%%%%%%%%%%%%%%%%%%%%%%%%%%%%%%%%%%%%%%%%%%%%%%%%%%%%%%%%%%%%%%%%%

\noindent \textbf{Study of $\textnormal{H}^0_3$.}\\
Now we count the number of words with weight $w=4$. In this case, the first condition set of $J_w$ becomes:
$$\left\{\begin{array}{l}
z_1+z_2+z_3+z_4=0\\
x_1 z_1+x_2 z_2+x_3 z_3+x_4 z_4=0\\
y_1 z_1+y_2 z_2+y_3 z_3+y_4 z_4=0
\end{array}\right.$$
We notice that this is a linear system in $z_i$. We first choose $4$ points
$P_i=(x_i,y_i)$ on $\He$ and then we calculate the number of solutions in $z_i$'s.
The coefficient matrix is
$$\left(\begin{array}{cccc}
1 & 1 & 1 & 1\\
x_1 & x_2 & x_3 & x_4\\
y_1 & y_2 & y_3 & y_4
\end{array}\right).$$
This matrix cannot have rank $1$. If the rank is $2$, this means that all $P_i$'s lie on a same line. The vector space of solutions has dimension $2$, so that we have
$q^4-4(q^2-1)-1$ solutions in $z_i$'s (we have to exclude the zero solution
and solutions with one $z_i=0$).

Otherwise, the rank is $3$. In this case, we have $3$ points on a same line,
say $P_1,P_2,P_3$, if and only if we have a square submatrix of order $3$
whose determinant is $0$, but this implies that $z_4=0$, which is not admissible.
If we choose $4$ points such that no $3$ of them lie on a same line,
all $z_i$'s can be non-zero and we get a codeword. The vector space of solutions has
dimension $1$, so that we have $q^2-1$ solutions in $z_i$'s (we have to exclude
the zero solution). 

If the rank is $2$, the total number of solutions (in $x_i,y_i,z_i$) is
$$\left(q^2\binom{q}{4}+(q^4-q^3)\binom{q+1}{4}\right)(q^4-4q^2+3).$$
If the rank is $3$, the total number of solutions (in $x_i,y_i,z_i$) is
$$\left(\binom{q^3}{4}
-q^2\binom{q}{3}(q^3-q)-(q^4-q^3)\binom{q+1}{3}(q^3-q-1)+\right.$$
$$\left.-q^2\binom{q}{4}-(q^4-q^3)\binom{q+1}{4}\right)(q^2-1).$$
Putting together, we get the total number of codewords of weight $4$ of
$\textnormal{H}^0_3$:
$$A_4=\left(\binom{q^3}{4}-q^2\binom{q}{3}(q^3-q)-
(q^4-q^3)\binom{q+1}{3}(q^3-q-1)\right)(q^2-1)+$$
$$+\left(q^2\binom{q}{4}+(q^4-q^3)\binom{q+1}{4}\right)(q^4-5q^2+4).$$

\noindent Doing the calculations we obtain the first part of Theorem \ref{teo.H3}.\\

%%%%%%%%%%%%%%%%%%%%%%%%%%%%%%%%%%%%%%%%%%%%%%%%%%%%%%%%%%%%%%%%%%%%%%%%
%%%%%%%%%%%%%%%%%%%%%%%%%%%%%%%%%%%%%%%%%%%%%%%%%%%%%%%%%%%%%%%%%%%%%%%%
%%%%%%%%%%%%%%%%%%%%%%%%%%%%%%%%%%%%%%%%%%%%%%%%%%%%%%%%%%%%%%%%%%%%%%%%

\noindent \textbf{Study of $\textnormal{H}^1_3$.}\\
We count the number of words with weight $w=4$. In this case, the first condition set of $J_w$ becomes:
$$\left\{\begin{array}{l}
z_1+z_2+z_3+z_4=0\\
x_1 z_1+x_2 z_2+x_3 z_3+x_4 z_4=0\\
y_1 z_1+y_2 z_2+y_3 z_3+y_4 z_4=0\\
x_1^2 z_1+x_2^2 z_2+x_3^2 z_3+x_4^2 z_4=0
\end{array}\right.$$
As above, we first choose $4$ points $P_i=(x_i,y_i)$ on $\He$ and then
we calculate the number of solutions in $z_i$'s. The coefficient matrix is
\begin{equation}\label{mat.1}
  \qquad\qquad\qquad\qquad\qquad   \left(\begin{array}{cccc}
1 & 1 & 1 & 1\\
x_1 & x_2 & x_3 & x_4\\
y_1 & y_2 & y_3 & y_4\\
x_1^2 & x_2^2 & x_3^2 & x_4^2
\end{array}\right).
\end{equation}
Now we study the rank of the matrix according to ``v-blocks''.

If all $x_i$'s are equal, we have $4$ points on a vertical line; the rank is $2$
(see below) and the number of codewords is (see case $\textnormal{H}^0_3$)
$$q^2\binom{q}{4}(q^4-4q^2+3).$$

If only three $x_i$'s are equal, we have $3$ points on a vertical line and another one outside, but this configuration is impossible for $\textnormal{H}^0_3$ (that is, we do not have codewords associated to it), and it is also impossible for
$\textnormal{H}^1_3$, since $\textnormal{H}^1_3\subset\textnormal{H}^0_3$.

If we have two pairs of equal $x_i$'s  (for instance, $x_1=x_2\ne x_3=x_4$), we can have codewords. In this case, we deduce 
$$
z_1+z_2=0,z_3+z_4=0,$$
$$
z_1(y_1-y_2)+z_3(y_3-y_4)=0,
$$
so that we have $\binom{q^2}{2}$ ways to choose $\{x_1,x_3\}$, $\binom{q}{2}$ ways to choose $\{y_1,y_2\}$, $\binom{q}{2}$ ways to choose $\{y_3,y_4\}$, $q^2-1$ ways to choose $z_1$, this  determines all $z_i$. The number of codewords in this case is
$$
\frac{q^4(q^2-1)^2(q-1)^2}{8}.
$$

If only two $x_i$'s are equal, say $x_1=x_2$, we can show that we have $z_1+z_2=0$, $z_3=0,z_4=0$, which is not admissible.

If we have all $x_i$'s distinct, the submatrix
$$\left(\begin{array}{cccc}
1 & 1 & 1 & 1\\
x_1 & x_2 & x_3 & x_4\\
x_1^2 & x_2^2 & x_3^2 & x_4^2
\end{array}\right)$$
has rank $3$, but if the whole matrix (\ref{mat.1}) has rank $4$ we can only have the zero solution, which is not admissible. Thus, (\ref{mat.1})  must have rank $3$, that is, the $y_i$'s row must be linearly dependent on the other rows. This means that
$$\exists a,b,c\in\FQ\quad|\quad\forall i=1,\ldots,4\quad y_i=ax_i^2+bx_i+c,$$
that is, all $P_i$'s lie on a same parabola (or on a same non-vertical line, when $a=0$). In this case, the number of codewords is
$$
(q^2-1)\sum_{k=4}^{2q}N_k\binom{k}{4},
$$
where $N_k$ is the number of parabolas and non-vertical lines that intersect $\He$ in exactly $k$ points.

\noindent Putting all together we get $A_4$, that is, the second part of Theorem \ref{teo.H3}.\\

%%%%%%%%%%%%%%%%%%%%%%%%%%%%%%%%%%%%%%%%%%%%%%%%%%%%%%%%%%%%%%%%%%%%%%%%
%%%%%%%%%%%%%%%%%%%%%%%%%%%%%%%%%%%%%%%%%%%%%%%%%%%%%%%%%%%%%%%%%%%%%%%%
%%%%%%%%%%%%%%%%%%%%%%%%%%%%%%%%%%%%%%%%%%%%%%%%%%%%%%%%%%%%%%%%%%%%%%%%

\noindent \textbf{Study of $\textnormal{H}^2_3$.}\\
We count the number of words with weight $w=4$. In this case, the first condition set of $J_w$ becomes:
$$\left\{\begin{array}{lll}
z_1+z_2+z_3+z_4=0\\
x_1 z_1+x_2 z_2+x_3 z_3+x_4 z_4=0\\
y_1 z_1+y_2 z_2+y_3 z_3+y_4 z_4=0\\
x_1^2 z_1+x_2^2 z_2+x_3^2 z_3+x_4^2 z_4=0\\
x_1 y_1 z_1+x_2 y_2 z_2+x_3 y_3 z_3+x_4 y_4 z_4=0
\end{array}\right.$$
As above, we first choose $4$ points $P_i=(x_i,y_i)$ on $\He$ and then
we calculate the number of solutions in $z_i$'s. The coefficient matrix is
\begin{equation}
 \qquad\qquad\qquad\qquad\qquad\left(\begin{array}{cccc}\label{mat.2}
1 & 1 & 1 & 1\\
x_1 & x_2 & x_3 & x_4\\
y_1 & y_2 & y_3 & y_4\\
x_1^2 & x_2^2 & x_3^2 & x_4^2\\
x_1 y_1 & x_2 y_2 & x_3 y_3 & x_4 y_4
\end{array}\right).
\end{equation}
Now we study the rank of the matrix according to ``v-blocks''.

If all $x_i$'s are equal, we have $4$ points on a vertical line; the rank is $2$
(see below) and the number of codewords is (see case $\textnormal{H}^1_3$)
$$q^2\binom{q}{4}(q^4-4q^2+3).$$

If only three $x_i$'s are equal, we have $3$ points on a vertical line and another one outside, but this configuration is impossible (as above).

If we have two pairs of equal $x_i$'s (for instance, $x_1=x_2\ne x_3=x_4$), we can deduce
$$
z_1+z_2=0,z_3+z_4=0,
$$
and then
$$
\left\{\begin{array}{lll}z_1(y_1-y_2)+z_3(y_3-y_4) & = & 0 \\
x_1 z_1(y_1-y_2)+x_3 z_3(y_3-y_4) & = & 0\end{array}\right.
$$
but this system in the unknown {\small{$y_1-y_2,y_3-y_4$}} has determinant  {\small{$z_1 z_3(x_3-x_1)\ne 0$}},
so that $y_1=y_2$, which is impossible.

If only two $x_i$'s are equal, say $x_1=x_2$, we can show that we have $z_1+z_2=0$, $z_3=0,z_4=0$, which is not admissible.

If we have all $x_i$'s distinct, the submatrix
$$\left(\begin{array}{cccc}
1 & 1 & 1 & 1\\
x_1 & x_2 & x_3 & x_4\\
x_1^2 & x_2^2 & x_3^2 & x_4^2
\end{array}\right)$$
has rank $3$, but if the whole matrix (\ref{mat.2}) has rank $4$ we can only have the zero solution, that is not admissible. Thus, (\ref{mat.2}) must have rank $3$, that is, the $y_i$'s and $x_i y_i$ rows must be linearly dependent on the other rows. This means that $y=ax^2+bx+c$ and $xy=dx^2+ex+f$, then $ax^3+(b-d)x^2+(c-e)x-f=0$. But this equation can have at most $3$ distinct solutions, and we need $4$, then we must have $a=0,
b=d,c=e,f=0$, that is, $y=bx+c$: all $P_i$'s lie on a same non-vertical line, and the number of codewords is
$$(q^4-q^3)\binom{q+1}{4}(q^2-1).$$

\noindent Putting all together we get $A_4$, that is, the last part of Theorem \ref{teo.H3}.\\

%%%%%%%%%%%%%%%%%%%%%%%%%%%%%%%%%%%%%%%%%%%%%%%%%%%%%%%%%%%%%%%%%%%%%%%%
%%%%%%%%%%%%%%%%%%%%%%%%%%%%%%%%%%%%%%%%%%%%%%%%%%%%%%%%%%%%%%%%%%%%%%%%
%%%%%%%%%%%%%%%%%%%%%%%%%%%%%%%%%%%%%%%%%%%%%%%%%%%%%%%%%%%%%%%%%%%%%%%%

\noindent \textbf{Study of $\textnormal{H}^0_4$.}\\
We count the number of words with weight $w=5$. We have a linear system in $z_i$ with
a $(6\times5)$ matrix. If its rank is $5$, we can only have the zero solution, which is not admissible. Thus, its rank must be at most $4$; this means that we have at least $2$ relationships of linear dependency, say
$$
\left\{\begin{array}{l}
xy=a+bx+cy+dx^2 \\ y^2=e+fx+gy+hx^2
\end{array}\right.
$$
we need to find $5$ points on the intersection of $2$ different conics, but this means
that the $2$ conics must be degenerate, they must have a common line, and all $5$ points
belong to this line. We could distinguish between vertical lines and non-vertical lines,
but in both cases the rank of the matrix is exactly $3$. So, the number of codewords is
$$A_5=\left((q^4-q^3)\binom{q+1}{5}+q^2\binom{q}{5}\right)(q^4-5q^2+4).$$

\noindent Doing the calculations we obtain the first part of Theorem \ref{teo.H4}.\\

%%%%%%%%%%%%%%%%%%%%%%%%%%%%%%%%%%%%%%%%%%%%%%%%%%%%%%%%%%%%%%%%%%%%%%%%
%%%%%%%%%%%%%%%%%%%%%%%%%%%%%%%%%%%%%%%%%%%%%%%%%%%%%%%%%%%%%%%%%%%%%%%%
%%%%%%%%%%%%%%%%%%%%%%%%%%%%%%%%%%%%%%%%%%%%%%%%%%%%%%%%%%%%%%%%%%%%%%%%

\noindent \textbf{Study of $\textnormal{H}^1_4$, $\textnormal{H}^2_4$, $\textnormal{H}^3_4$.}\\
To count the number of words with weight $w=5$, we remember that
$$\textnormal{H}^0_4\supseteq\textnormal{H}^1_4\supseteq\textnormal{H}^2_4\supseteq
\textnormal{H}^3_4\supseteq\textnormal{H}^0_5$$
and the first and the last code have all words with weight $5$ corresponding to $5$
points on a line. We notice that for a vertical line the rank of the matrix is $3$,
while for a non-vertical line the rank of the matrix is $4$. So, the number of
codewords is
$$A_5=q^2\binom{q}{5}(q^4-5q^2+4)+(q^4-q^3)\binom{q+1}{5}(q^2-1).$$

\noindent Doing the calculations we obtain the last part of Theorem \ref{teo.H4}.

%==================================================================

\section{Computational verification}
\label{comp}
We have verified all our formulas for the number of small weight
codewords, that is, Theorems \ref{teo1}, \ref{teo2}, \ref{teo:d+1},
\ref{teo2:d+1}, \ref{teo3:d+1}, \ref{teo.H3}, \ref{teo.H4}.\\

The verification has been done by computing a \Gr\ basis of ideal
$J_w$ for the corresponding case, as in Subsection \ref{parolemin}, and counting
the number of its solutions as in \cite{CGC-cry-art-caruso08}.
As software package we used Singular and MAGMA \cite{CGC-alg-misc-GreuelPfistSchn07,CGC-MAGMA} and all our
programmes and their digital certificates are available under request.

%==================================================================

\section{Conclusions and open problems}
\label{conc}
The so-called first-phase codes have nice geometric properties that
allow their study, as first realized in \cite{CGC-cod-misc-pell06} 
and \cite{CGC-cod-misc-salpell06}. In particular, the fact
that all minimum-weight codewords lie on intersections of lines and
$\He$ is essential. Recent research has widened this
approach to intersection with degree-$2$ and degree-$3$ curves
\cite{CGC-cod-art-couvreur2011dual,CGC-cod-art-ballico2012geometry}, unfortunately without reaching
an exact formula for higher weights. We believe that only complete
classifications of intersections of $\He$ and higher degree
curves can lead to the determination of the full weight distribution
of first-phase Hermitian codes.
We invite the reader to pursue this approach further.

As regards the other phases, it seems that only a part of the second
phase can be described in a similar way.
Therefore, probably a radically different approach is needed for
phase-$3,4$ codes in order to determine their weight distribution completely.
Alas, we have no suggestions as to how reach this.

%==================================================================

\section*{Acknowledgements}
\label{ack}

This work was partially presented in 2006 at two conferences: \cite{CGC-cod-misc-pell06}
and \cite{CGC-cod-misc-salpell06},
in 2010 at \textit{Workshop on Coding \& Cryptography}, Cork (Ireland),
in 2011 at the \textit{Mzuni Math Workshop}, Mzuzu University (Malawi)
and at WCC \cite{CGC-cod-art-marcpell2011weights}.\\

The seminal idea behind our starting point (Proposition \ref{proptot})
was in the second author's PhD thesis \cite{CGC-cd-phdthesis-marco},
which was studied in deep and generalized in the first author Master's
thesis and PhD thesis \cite{CGC-alg-tesi3-marcolla08,CGC-cd-phdthesis-marcolla}.
The first two authors would like to thank their supervisor, the third author.\\

The authors would like to thank C. Traverso for seminal discussions in
2005 on the relation between
small-weight codewords and variety points.

%==================================================================

\bibliography{RefsCGC}

\end{document}